\documentclass[12pt]{article}
\usepackage{graphicx,a4}
\usepackage{fancyhdr,amsfonts,amsmath,amssymb,mathrsfs,amsthm,epsfig,color}

\usepackage{hyperref}
\definecolor{Red}{cmyk}{0,1,1,0}

\definecolor{verde}{cmyk}{1,0,1,0}

\definecolor{azul}{cmyk}{1,1,0,0}

\usepackage{tocloft}

\newcommand{\horrule}[1]{\rule{\linewidth}{#1}}

\newtheorem{theorem}{Theorem}
\newtheorem{corollary}[theorem]{Corollary}
\newtheorem{definition}[theorem]{Definition}

\newtheorem{lemma}[theorem]{Lemma}

\newtheorem{proposition}[theorem]{Proposition}
\newtheorem{remark}[theorem]{Remark}

\global\newcount\numsec\global\newcount\numfor

\usepackage{mathtools}

\begin{document}
\begin{center}
\vspace*{0.5 cm}
%
%
{\large\bf
Phase Transitions in One-dimensional
Translation Invariant Systems:
a Ruelle Operator Approach.
}
%
%
%
%
\vspace*{0.5cm}\\
    Leandro Cioletti$^{\dagger}$ and
    Artur O. Lopes$^{\ddagger}$
\\
 \vskip 3mm
$\dagger$
{\footnotesize Universidade de Brasilia, 70910-900 Brasilia-DF, Brazil}
\\
$\ddagger$
{\footnotesize UFRGS,  91.500 Porto Alegre-RS, Brazil}
\end{center}
\vskip 5mm
%
%
%
%

\begin{abstract}

We consider a family of potentials $f$, derived from the Hofbauer
potentials, on the symbolic space
$\Omega=\{0,1\}^\mathbb{N}$  and the shift mapping $\sigma$ acting on it.
A Ruelle operator framework
is employed to show there is a phase transition when the temperature
varies in the following senses: the pressure is not analytic, there
are multiple eigenprobabilities for the dual of the Ruelle operator,
the DLR-Gibbs measure is not unique and finally
the Thermodynamic Limit is not unique.
Additionally, we explicitly calculate the critical points for these
phase transitions. Some examples which are not of
Hofbauer type are also considered.
The non-uniqueness of the Thermodynamic
Limit is proved by considering
a version of a Renewal Equation. We also show that
the correlations
decay polynomially and compute the decay ratio.
\end{abstract}
\vskip 1mm
{\footnotesize
MSC: 82B20, 82B41, 82B26, 37D35, 37A35, 37A50, 37A60.
\newline
Keywords: DLR-measures, Gibbs Measures, Thermodynamic Limits, Ruelle Operator,
Dual measures, Phase Transition.
}

\setcounter{tocdepth}{1}
\noindent\horrule{1.0pt}
\tableofcontents
\vspace*{0.3cm}
\noindent\horrule{1.0pt}
%
%

\section{Introduction}\label{sec1}

In this work we are interested in the study of the phase
transition phenomenon, as temperature varies,
for some systems defined on the one-dimensional lattice
$\mathbb{N}$.
Our study focus mainly on the models which are generalizations of the
so called Hofbauer models on $\Omega=\{0,1\}^{\mathbb{N}}$
(see \cite{Hof} and \cite{Wal}).
The Hofbauer models, besides being interesting mathematical objects,
are also related to the
B. Fedelhorf and M. Fisher models (see \cite{FF} and \cite{F})
which appear in the Physics literature (see also \cite{Wang1} and \cite{Wang2}).
They are also studied in
\cite{L1}, \cite{FL}, \cite{IT} and \cite{L0}, and are
 associated to a family of continuous potentials.
They are also singled out for being fixed points of a certain renormalization
mapping (see \cite{BLL}, \cite{BL} and \cite{BL1}).

The analysis of
Phase Transitions in Thermodynamic Formalism is a problem which can be
understood in several different settings (see \cite{L2} \cite{Sar1}
\cite{Sar2} \cite{Sar3} \cite{L8} \cite{L7} \cite{Jap} \cite{CR}).
Loosely speaking, in most cases a phase transition is defined
by an abrupt qualitative change of the model
in terms of some of its parameters (in general, the inverse of the
 temperature) in a neighborhood of
a special value, called critical point or phase transition point.

We will present here several results on phase transitions for the one
dimensional lattice $\mathbb{N}$. Concepts and results which are common in
Statistical Mechanics are explored here in the setting of Thermodynamic
Formalism.
We would like to mention the very interesting work \cite{Sarig1}
where this unifying point of view is explored and nice results are carefully
presented.
We refer the reader to the references \cite{PP}, \cite{Bow} and \cite{Walk} for basic results in
Thermodynamic Formalism and Ergodic Theory and \cite{Rue}, \cite{Kell}, \cite{BC}, \cite{EFS}, \cite{Bov}, \cite{Ellis} and
\cite{Geo}, \cite{Isra} for basic results in Statistical Mechanics.
In the companion paper \cite{CL} the authors presented in detail several
concepts and results which are used here.

To give precise statements of the main results of this paper
we introduce some notations and definitions.
We consider the Bernoulli space or sometimes
called symbolic space $\Omega=\{0,1\}^\mathbb{N}$ and the full
left shift $\sigma:\Omega\to\Omega$ acting on this symbolic space.

 The main focus in the paper are potentials defined
on the set  $\Omega=\{0,1\}^\mathbb{N}$  instead of
the usual space $\Omega=\{0,1\}^\mathbb{Z}$. This is the setting where one can apply the Ruelle operator formalism (see \cite{Rue}). The final result
applies to $\Omega=\{0,1\}^\mathbb{Z}$  as we will describe in full detail in the last section. It is just a technical issue of using coboundary functions to go from one setting to the other.

A potential is a continuous function $f: \Omega \to \mathbb{R}$ which
describes the interaction of spins in the lattice $\mathbb{N}$ .
We denote by $\mathcal{M}_{1}(\sigma)$ the set of
invariant probabilities measures (over the Borel sigma algebra of $\Omega$) under
$\sigma$.

\begin{definition}[Pressure]
For a continuous potential $f:\Omega \to \mathbb{R}$ the Pressure
of $f$ is given by
$$
	P(f)
	=
	\sup_{\mu\in \mathcal{M}_{1}(\sigma)}
		\left\{
			h(\mu) + \int_{\Omega} f\, d\mu
		\right\},
$$ where $h(\mu)$ denotes the
Shannon-Kolmogorov entropy of $\mu$ (see \cite{Walk} for definition).
\end{definition}

\begin{definition}[Equilibrium State]
A probability measure $\mu\in \mathcal{M}_1(\sigma)$
is called an equilibrium state for $f$ if
$$
	h(\mu) + \int_{\Omega} f\, d\mu=P(f).
$$
\end{definition}
\noindent{\bf Remark.}
If $f$ is continuous there always exists at least one equilibrium state (see \cite{WAL7}).
For any potential in the class W$(X,T)$  or Bow$(X,T)$ the equilibrium state is
unique (see \cite{Wal0}). When $f$ is assumed to be only a continuous
function there are examples where
we have more than one equilibrium state for $f$.
The existence of more than one equilibrium state is a possible meaning for phase transition.
An example is given in \cite{Hof} and
here we present some other examples.
As we mention before there are several definitions of phase transition.
They are not necessarily equivalents. On this paper
we investigate five types of phase transition (listed below) and exhibit a
continuous potential $f$ where these different notions coincides on
a single critical value $\beta_c>0$.

\begin{enumerate}\label{definicoes-phase-transition}
\item The function $p(\beta)=P(\beta\, f)$ is not analytic at $\beta=\beta_c$

\item There are more than one equilibrium state, that is, at least two
probability measures maximizing $h(\mu) + \beta_c \int_{\Omega} f\, d \mu$.

\item The dual of Ruelle operator (to be defined later) has more than one
eigenprobability for $\beta_c f$

\item There exist more than one DLR (to be defined later) probabilities measures for the
potential $\beta_c f$.

\item There is more than one Thermodynamic Limit (to be defined later).
\end{enumerate}

We refer the reader to \cite{CL} where the equivalences among
the definitions $3,4$ and $5$ are proved
for Bowen potentials. As far as we know there is no general
theory on Phase Transitions and examples
are handled in a case by case basis.

Decay of correlation of exponential
type (for a large class of observable functions)
occurs for the equilibrium probability of a
H\"older potential. By the other hand, in some cases
where there is  phase transition (not H\"older), for
the equilibrium  probability (at the transition temperature)
one gets polynomial decay of correlation. We show in
Section \ref{decay} that this is indeed the
case for the example we analyze in the paper.
In the proof we use Renewal Theory.

There are interesting questions related to what happens
near this critical point and regarding the properties of the model at it.
For example, the problem of maximizing probabilities
and selection or non-selection at
zero temperature in distinct models were
analyzed in \cite{BLL1}, \cite{Cha},
\cite{van}, \cite{Lep3}, \cite{GT} \cite{BCLMS} and \cite{LMMS}.
In some cases, there is more than one selected ground state.
For the potentials considered here
(the Double Hofbauer potentials) we analyze questions about
selection or non-selection at a positive critical temperature
for both symmetric and asymmetric cases. The term
``phase transition'' shall hereby be solely employed when
the critical temperature is strictly positive.
\\
\\

Regarding to the first notion of phase transition
given a continuous potential $f$ a natural question is
whether the Pressure $P(\beta\, f)$ is differentiable or even analytic
as a function of $\beta$, the inverse temperature.
In such generality this question is very hard but if
$f$ is a H\"older potential then
the mapping $\beta\mapsto P(\beta f)$
is real analytic function for any $\beta>0$
(see \cite{PP} and \cite{Bow}).

In the sequel we introduce the basic concepts we used in the
phase transition notions listed above.

\bigskip

 \begin{definition} Given a continuous function $f: \Omega \to \mathbb{R}$,
 consider the Ruelle operator (or transfer)
 $\mathcal{L}_f:C(\Omega)\to C(\Omega)$
 (for the potential $f$)
 defined in such way that for any continuous
 function $\psi:\Omega\to\mathbb{R}$ we have $\mathcal{L}_{f} (\psi)=\varphi$, where
$$
	\varphi(x)=\mathcal{L}_{f} (\psi)(x)
	=
	\sum_{ y\in\Omega;\, \sigma(y)=x} e^{ f(y)} \, \psi(y).
$$
\end{definition}

\noindent{\bf Remark.}
When $f$ is a H\"older function then $\mathcal{L}_{f}$
sends the space of H\"older functions to itself.

This operator (which is also called transfer operator) is
a very helpful tool in the analysis of equilibrium states
in Thermodynamic Formalism. One important issue
is the existence or not of an strictly positive continuous
eigenfunction for such operator.

\begin{definition}
The dual operator $\mathcal{L}_{f}^*$ acts on the space
of probability measures. It sends a probability measure $\mu$
to a probability measure $\mathcal{L}_{f}^*(\mu)=\nu$
defined in the following way:
the probability measure $\nu$ is unique
probability measure satisfying
$$
	\int_{\Omega} \, \psi d\, \mathcal{L}_{f}^* (\mu)
	=
	\int_{\Omega} \psi\, d\nu
	=
	\int_{\Omega} \mathcal{L}_{f}(\psi) d \mu
$$
for any continuous function $\psi$.
\end{definition}

\begin{definition}[Gibbs Measures]\label{ft-gibbs-measures-lam}
Let $f:\Omega\to\mathbb{R}$ be a continuous function.
We call a probability measure $\nu$ a Gibbs probability for $f$
if there exists a positive
$\lambda>0$ such that $\mathcal{L}_f^*(\nu) =\lambda\, \nu$.
We denote the set of such probabilities by $\mathcal{G}^*(f)$.
\end{definition}

In general, for any continuous function $f$ the set
$\mathcal{G}^*(f) \ne \emptyset$.
We should remark that a probability measures on $\mathcal{G}^*(f)$
is not necessarily translation invariant, even if $f$ is Lipchitz or H\"older.

For a H\"older potential $f$ there exist a value $\lambda>0$ (the spectral radius) which is a common
eigenvalue for both Ruelle operator and its dual (see \cite{PP}).
The eigenprobability $\nu$ associated to the maximal $\lambda$ is
unique.
This probability $\nu$ (which is unique) is a Gibbs state according
to the above definition.

A natural question is: for a Holder continuous potential $f$ is it possible to have different eigenprobabilities $\nu_1,\nu_2\in \mathcal{G}^*(f)$ for $\mathcal{L}_f^*$, associated to different eigenvalues $\lambda_1$ and $\lambda_2$? This is not possible due to properties of the involution kernel (see \cite{LMMS}). These two eigenprobabilities would determine via the involution kernel two positive eigenfunctions for another Ruelle operator $\mathcal{L}_{f^*}$, with the eigenvalues  $\lambda_1$ and $\lambda_2$ (see \cite{GLP}), where $f^*$ is the dual potential for $f$. This is not possible (see for instance \cite{PP} or Proposition 1 in \cite{LMMS}). Anyway, eigendistributions for $\mathcal{L}_f^*$  may exist (see \cite{GLP}).

This eigenvalue $\lambda$ is the spectral radius of the operator
$\mathcal{L}_f$.
If $\mathcal{L}_f(\varphi)= \lambda \varphi$
and $\mathcal{L}_f^*(\nu)= \lambda \nu$, then up to
normalization (to get a probability measure) the
probability measure $\mu=\varphi\, \nu$ is the
equilibrium state for $f$.

Because of this last mentioned fact
the Gibbs probability $\nu$ helps
to identify the equilibrium probability $\mu$.
If $\varphi$ is constant equal to $1$ then the
Gibbs probability measure
is the equilibrium probability.

These are non trivial results
and the detailed proofs can be found in \cite{PP}.
Several properties of the equilibrium state $\mu$ (mixing, exponential decay of
correlations, central limit theorem, etc...) are obtained
using this formalism in the H\"older
case (see \cite{PP}).
The bottom line is:
several nice properties of the Gibbs state
follow from the use of  Ruelle operator properties.
All these are extended to the equilibrium state
via this formalism.

Under the square summability of the variation
for normalized potentials there exist just
one eigenprobability (see \cite{JO} and \cite{JOP})
and therefore there is no
phase transition in this case in this sense.

When there exists a positive continuous eigenfunction
for the Ruelle operator (of a continuous potential $f$)
it is unique (the proof in \cite{PP} works for a continuous potential $f$). We remark that
for a general continuous potential may not exist a
positive continuous eigenfunction.
We present some examples here. On the other hand,
eigenprobabilities always exists for a
continuous potential since the space $\Omega$ is compact (see Theorem \ref{w2}).
If for a continuous potential $f$ there exists an eigenfunction $\varphi$ and
an eigenprobability $\nu$, then $\mu=\varphi\, \nu$
defines an equilibrium state for $f$ (see Section 2 in \cite{PP}).

For a H\"older potential $f$ we do not have phase transintion in any above
described notions. One heuristic reason is in the H\"older case
the influence over the state in fixed site $n$ in the lattice,
by any other site decays exponentially fast with respect to the distance
between them.

\medskip

We will consider in this paper some continuous potentials
defined by P. Walters in \cite{Wal} on the Bernoulli space
$\{0,1\}^\mathbb{N}$. The values of a potential $f$
in this class are defined just by
the first strings of zeroes and ones.
To be more precise consider four sequence of real numbers
$a_n,b_n,c_n,d_n$ and constants $a,b,c,d$.
The potential $f$ in this class satisfies
\begin{align*}
f(0^n1z)= a_n,\quad
f(0^{\infty})=a,\quad
f(10^n1z)=d_n,\quad
f(10^{\infty})=d,\quad
\\
f(01^n0z)=b_n,\quad
f(01^{\infty})=b,\quad
f(1^n0z)=c_n,\quad
f(1^{\infty})=c,\quad
\end{align*}
and is assumed that $a_n \to a$, $b_n \to b$, $c_n \to c$ and $d_n \to d$.

The existence of positive {\bf continuous} eigenfunction
for $f$ is guaranteed (see Theorem 3.1 in \cite{Wal}) as long as
the following inequality is satisfied
\begin{equation}
\label{ee1}
	1<
	\frac{1}{e^{2\max(a,c)}}
	\left[
		e^{d_1}+\sum_{j=1}^{\infty} e^{d_{j+1}}\frac{e^{a_2+\ldots+a_{j+1}}}{e^{j\max(a,c)}}
	\right]
	\left[
		e^{b_1}+\sum_{j=1}^{\infty} e^{b_{j+1}}\frac{e^{c_2+\ldots+c_{j+1}}}{e^{j\max(a,c)}}
	\right].
\end{equation}
In this case (see Theorem 3.5 in \cite{Wal}) the eigenvalue $\lambda$ for the Ruelle operator satisfies:
\begin{equation}
\label{ee2}
	1=
	\frac{1}{\lambda^2}
	\left[
		e^{d_1}+\sum_{j=1}^{\infty} e^{d_{j+1}}\frac{e^{a_2+\ldots+a_{j+1}}}{\lambda^j }
	\right]
	\left[
		e^{b_1}+\sum_{j=1}^{\infty} e^{b_{j+1}}\frac{e^{c_2+\ldots+c_{j+1}}}{\lambda^j}
	\right].
\end{equation}
There are continuous potentials $f$
of the above form such that the condition (\ref{ee1}) is not satified.
Such potentials provide examples where the potential $f$
is continuous but the Ruelle operator associated to it
do not have positive continuous eigenfunction.
When the eigenfunction do exists an explict formula for it is given in \cite{Wal}
(see page 1341).
We point out that all of the above formulas are analytical expressions and
even in the case where the r.h.s of (\ref{ee1}) is equal to $1$ an
explicit eigenfunction which is not continuous can be obtained
(because, for instance is $\infty $ just in the points $0^\infty$ and $1^\infty$
but finite elsewhere).
This extended sense of eigenfunction
will be considered here in {\bf the critical point} in some of our examples.

In \cite{BLM} it is analyzed the zero temperature limit for such family of potentials.
\medskip

We will define in the next section the double Hofbauer model which is obtained from a certain potential $g:\{0,1\}^\mathbb{N}=\Omega \to \mathbb{R}$, depending of two fixed parameters $\gamma $ and $\delta$.

The potential view of renormalization is presented in \cite{BLL}. The double Hofbauer potential is a fixed point for a renormalization operator. In this sense it is a model of special interest. We will explain briefly this point.

We define $H : \Omega=\{0,1\}^\mathbb{N} \to \Omega $ by: for $c_1\geq 2$
$$H((\underbrace{0,...,0}_{c_1},\underbrace{1, ... ,1}_{c_2}\underbrace{0,...,0}_{c_3},1,...))=
(\underbrace{0,...,0}_{2c_1},\underbrace{1,...,1}_{c_2}
\underbrace{0,...,0}_{c_3},1,\ldots),$$
and
$$
H((\underbrace{1,\ldots,1}_{c_1},\underbrace{0,\ldots,0}_{c_2}\underbrace{1,
\ldots,1}_{c_3},1,\ldots))=
(\underbrace{1,\ldots,1}_{2c_1},\underbrace{0,\ldots,0}_{c_2}
\underbrace{1,\ldots,1}_{c_3},1,\ldots).
$$

We define the renormalization operator $\mathcal{R}$  in the following way:
given the potential $V_1: \Omega \to \mathbb{R}$   we get  $V_2 = \mathcal{R}
(V_1)$ where for $x$ of the form $(\underbrace{0,...,0}_{c_1}\,1...)$, or $(\underbrace{1,...,1}_{c_1}\,1...)$, with $c_1\geq 2$:
$$V_2(x)=
V_1(\sigma(\, (H(x)))\,)+V_1(H(x)),$$
and for $x$ of the form $(01...)$ or $(10...)$ we set
$$V_2(x)=V_1(x).$$

This defines  $V_2 = \mathcal{R}
(V_1)$.

It is easy to see that for $\gamma $ and $\delta$ fixed  the corresponding
double Hofbauer potential $g$ is fixed for $\mathcal{R}$.
\medskip

In \cite{BLL} it is explained why this is the natural renormalization operator to be considered in the one-dimensional setting (which is inspired by the similar concept in Statistical Mechanics). It is more common in dynamics to consider the renormalization of the transformation dynamics (the M. Feigenbaum point of view) which is different from the reasoning described above.

\section{The Double Hofbauer Model}

Before present the model we need to introduce some notations.
We define two infinite collections of cylinder sets given by
\[
	L_n = \overline{\underbrace{000...0}_n 1}
	\quad\text{and}\quad
	R_n = \overline{  \underbrace{111...1}_n 0},
	\ \ \text{for all}\ n\geq 1.
\]
Note that these cylinders are disjoint and
$\cup_{n\geq 1}(L_n\cup R_n)=\Omega\setminus\{0^{\infty},1^{\infty}\}$.
To define the model we also need to fix two real numbers
$\gamma>1$ and $\delta>1$,
satisfying $\delta<\gamma$. Using these parameters
we can define a continuous potential $g_{\gamma,\delta}:\Omega\to\mathbb{R}$,
which is simply denote by $g$, in the following way: for any $x\in\Omega$
\[
g(x)
=
\begin{cases}
-\gamma \log \frac{n}{n-1}, &\text{if}\ x \in L_n,\ \text{for some}\ n\geq 2;
\\[0.1cm]
-\delta \log \frac{n}{n-1}, &\text{if}\ x \in R_n,\ \text{for some}\ n\geq 2;
\\[0.1cm]
-\log \zeta(\gamma), &\text{if}\ x \in L_1;
\\[0.1cm]
-\log\zeta(\delta), &\text{if}\ x \in R_1;
\\
0, &\text{if}\ x \in \{1^{\infty},0^{\infty}\},
\end{cases}
\]
where $\zeta(s)=\sum_{n\geq 1}1/n^s$.
By using the canonical identification of a point
of the symbolic space $\Omega$ with a point on the interval $[0,1]$
we can have a sketch of the graph of the Double Hofbauer potential
\begin{figure}[h!]
    \centering
   \includegraphics[scale=0.761,angle=0]{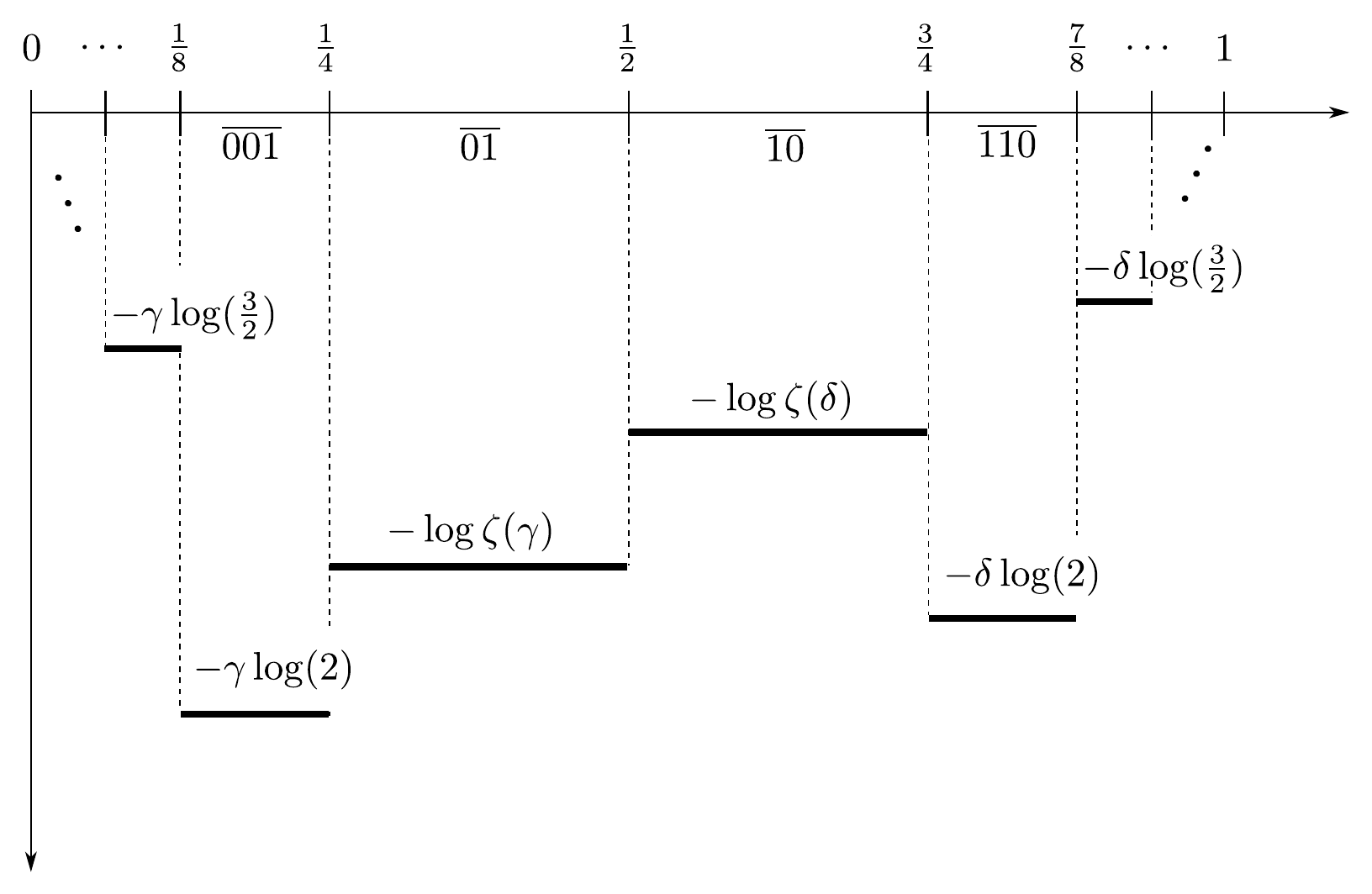}
   \caption{The Double Hofbauer potential represented on the interval $[0,1]$.}
\end{figure}

These potentials are particular cases of a more general class considered by
F.Hofbauer \cite{Hof} and P. Walters \cite{Wal}.
They are not H\"older pontetials and
have not summable variation (see \cite{Wal} or \cite{Sarig1}).
In fact, following the notation of the work \cite{Wal} (page 1325)
our potential are obtained by considering $a_n= -\gamma \log
\frac{n}{n-1}$, $a=0$, $c_n = -\delta \log \frac{n}{n-1}$, $c=0$, $b_n=  - \log
\zeta(\gamma)$, $d_n = - \log \zeta(\delta)$, $b= - \log \zeta(\gamma)$ and $d
= - \log \zeta(\delta)$, $n \in \mathbb{N}$.
For these choice of the sequences
$a_n,b_n,c_n,d_n$ and the constants $a,b,c,d$
Walters proved in \cite{Wal} that these potentials does not
belongs to W$(X,T)$ neither Bow$(X,T)$.
This fact follows from a simple application of the Theorem 1.1 page 1326 in \cite{Wal}.

When $\delta\neq \gamma$ there is a competition of two regimes.  This system has a much more complex behaviour that the one presented in the Fisher-Fedelhorf model described in \cite{FF}.

When $\gamma=\delta$ we say that the potential $g$ defined above is symmetric.
To avoid confusion we use the terminology Hofbauer model
(indexed by $\gamma$ which differs from the
Double Hofbauer model which is indexed by $\gamma$ and $\delta$)
for the family of potentials considered in \cite{L1}, \cite{BLL} and \cite{FL}.
There are, of course, some similarities between the two models.

For a while we will not assume that the potential is
symmetric. When we need such hypothesis we will make it clear.

\section[Phase Transition I. Non-differentiability of the Pressure]
{Phase Transition I.\\ Non-differentiability of the Pressure}

Note that the delta Dirac $\delta_{0^\infty} $ and $\delta_{1^\infty}$ are both
equilibrium states for $g$. On page 1341 in \cite{Wal}
it is presented the explicit expression for the
eigenfunction $\varphi_{\beta}$
associated to the eigenvalue $\lambda(\beta)$ for the
Ruelle operator of the potential $\beta g$.

One important issue is to show that the pressure
$p(\beta)$ of the potential $\beta\, g$
is equal to $ \log \lambda(\beta)$, where $\lambda(\beta)$
is the main eigenvalue of Ruelle operator for  $\beta\,g$.
This is the claim of Theorem \ref{w2}.
As we will see the Theorem \ref{w2} says that
$$
	\sup_{\mu\in \mathcal{M}_1(\sigma)}
		\left\{ h(\mu) + \beta\, \int_{\Omega} g\, d\mu
		\right\}
	=
	P(\beta\,g)
	\equiv
	p(\beta)
	=
	\log \lambda(\beta),
$$
where $\lambda(\beta)$ is the maximal eigenvalue
of $\mathcal{L}_{\beta\,g}$.
Moreover we can show that $P(1\, g)=0$
(see Theorems \ref{mainp} and \ref{w1}) and $P(\beta\, g)>0$ for
$\beta<1$.

\medskip

\noindent{\bf Remark.}
The function $\beta\mapsto P(\beta\,g)$ is non-increasing
and therefore for $\beta>1$ we have that
$P(\beta\, g)\leq 0$.
Since $h(\mu) + \beta\, \int_{\Omega} g\, d\mu=0$ for
$\mu=\delta_{0^\infty}$ and
$\mu=\delta_{1^\infty}$
we have that $P(\beta\,g)=0$ for $\beta>1$.

\medskip

We are interested in to determine equilibrium states
$\mu_{\beta}$ for the family of potentials $\beta\, g$,
when $\beta$ approaches $1$ from below.
If $0<\beta<1$ we have that
$p(\beta)\equiv P(\beta\, g)>0$, because
this function is monotone increasing and $P(g)=0$
(see Theorem \ref{w1}).
These observations give rise to natural questions as:
is there a selection in the limit when $\beta\to 1$ ?
In other words,
does $\mu_{\beta}$ converges to $\delta_{0^\infty}$
or to $\delta_{1^\infty}$
when $\beta\to 1$?
In the negative case is still the weak limit of
$\mu_{\beta}$ a
convex combination of these two probability measures ?
In what follows we show some computations
that will help to answer these questions.

Using the general result described by Corollary 3.5
in \cite{Wal} in our
particular case we have that the eigenvalue
$\lambda(\beta)=\exp(P(\beta\,g))\equiv \exp(p(\beta))$
satisfies the following identity
\begin{equation} \label{maindouble}
1
=
\frac{
\sum_{n=1}^\infty
\frac{n^{-\gamma\, \beta}}{\lambda(\beta)^{n}}
}
{\zeta(\gamma)^\beta}
\,\,\,
\frac{
\sum_{n=1}^\infty
\frac{n^{-\delta\,\beta}}{\lambda(\beta)^{n}}
}
{\zeta(\delta)^{\beta}}.
\end{equation}

\bigskip

For different values of the parameter $\gamma$ the above family provides
examples where we have phase transition of type 1 and 2 as defined
on the page \pageref{definicoes-phase-transition}.
In all these examples the critical inverse temperature
$\beta_c$ can be
explicitly obtained and its value is
$\beta_c=1$.
To be more precise
if  $2>\gamma>\delta>1$ there exists just two ergodic
equilibrium probabilities at $\beta=1$:
the Dirac measure concentrated on $0^\infty$ and the Dirac
measure concentrated on $1^\infty$.
In the case  $\gamma,\delta>2$ there exists three ergodic equilibrium
probabilities at $\beta=1$: the Dirac measure concentrated on $0^\infty$,
the Dirac measure concentrated on $1^\infty$ and a probability measure
$\mu_1$ which gives positive probability to open sets
(which is described latter).

\bigskip

Question:
Since $\beta=1$ is a critical temperature for the Double Hofbauer model
it is interesting to know what is the asymptotic behavior
of $p(\beta)=P(\beta\,g)$ when $\beta\to 1$ with $\beta<1$.
Of course, the other
lateral limit is trivial because pressure vanish for $\beta>1$.

\bigskip

\noindent
Let $p_1(\beta)=\log (\lambda_1(\beta))$
denote the pressure for the Hofbauer model
associated to $\gamma>1$ and
$p_2(\beta)=\log(\lambda_2(\beta))$
be the pressure for the Hofbauer model associated to
$\delta>1$ according to \cite{L1}.
If $\gamma,\delta>1$ the asymptotics expansions
of $p_1(\beta)$ and $p_2(\beta)$, when $\beta\to 1$ from below, were determined
in \cite{L1} (Theorem A) and for the reader's convenience we
give the statement of the theorem below:

\begin{theorem}\label{ptL}
For the Hofbauer model with
parameter $\gamma$ we have:

\begin{itemize}
\item[a)]
\(
\zeta(\gamma)^{\beta}
=\sum_{n=1}^{\infty}
\frac{e^{-n\,p_1(\beta)}}{n^{\gamma \beta}},\,\,
\text{for}\,\,\beta<1,
 \ \text{for any}\ \gamma>1,
\)

\item[b)]
If $1<\gamma<2$, then, when
$\beta\leq 1,\ \beta\rightarrow 1$,
we get that
\[
	p_1(\beta)
	=
	\left(
	\frac{\zeta(\gamma)\,
	\log\zeta(\gamma)-\gamma\,\zeta^{\prime}(\gamma)}
	     {-\Gamma(1-\gamma)}
	\right)^{\frac{1}{\gamma-1}}
	(1 -\beta)^{\frac{1}{(\gamma -1)}}
	+\ \text{\textnormal{high order terms}},
\]
\item[c)]
If $2<\gamma<3$,  then, when
$\beta\leq 1,\ \beta\rightarrow 1$,
there is constant $A$ so that
\[
p_1(\beta)=
\frac{\zeta(\gamma)\,
\log\zeta(\gamma)-\gamma\,\zeta^{\prime}(\gamma)}
{\gamma\,\zeta^{\prime}(\gamma-1)}
(1-\beta)+A(1-\beta)^{\gamma-1}(1+o(1)).
\]
In this case the entropy of the probability
measure $\mu$ (the equilibrium state
for the Hofbauer model) is
\[
\frac{\zeta(\gamma)\,
\log\zeta(\gamma)-\gamma\,\zeta^{\prime}(\gamma)}
{\gamma\,\zeta^{\prime}(\gamma-1)}.
\]
\end{itemize}
The case $\gamma>3$ can be also analyzed but the formulas are more complex.
\begin{itemize}
\item[d)] when $\beta\rightarrow 1$, $3\leq m<\gamma<m+1$, we get the expansion
\[
p_1(\beta)=A_{1}(1-\beta)+A_{2}(1 -\beta)^{2}
+\cdots+A_{m-1}
(1-\beta)^{m-1}+(C+o(1))(1-\beta)^{\gamma-1}.
\]
for some constants $A_{1},\ A_{2},\ \ldots, A_{m-1}$ and $C$.
\end{itemize}
\end{theorem}
\noindent{\bf Remark.}
Obviously, $p_2(\beta)$ has similar properties.

\begin{theorem} \label{ld1} If $1<\delta< \gamma<2$, then,
$p'(\beta)=\frac{d}{d\beta}
\log \lambda(\beta)\to 0$, when
$\beta\to 1$, $\beta<1$. In  the  case  $\gamma>\delta>2$,
we have
\[
\lim_{\beta \to 1^{-}} p(\beta)
= \frac{1}{2}
	\left(
	\lim_{\beta \to 1^{-}} p_1(\beta)
	+
	\lim_{\beta \to 1^{-}} p_2(\beta)
	\right).
\]
Since $p(\beta)=0$ for $\beta>1$
there is a lack of analyticity of the pressure
$p(\beta)$ at $\beta=1$.

\end{theorem}

\begin{proof}
It is known from \cite{L1}  that
\begin{equation}\label{prob1}
1
=
\lambda_1(\beta)^{-1} \,
\frac{
1 + \sum_{n=2}^\infty
\frac{n^{-\gamma\,\beta}}{\lambda_1(\beta)^{n-1}}
}
{ \zeta(\gamma)^{\beta}}
\,=\,
\frac{
\sum_{n=1}^\infty
\frac{n^{-\gamma\, \beta}}{\lambda_1(\beta)^{n}}
}
{ \zeta(\gamma)^{\beta}} ,
\end{equation}
and
\begin{equation}\label{prob2}
1
=
\lambda_2(\beta)^{-1}
\frac{
1 + \sum_{n=2}^\infty
\frac{n^{-\delta\,\beta}}{\lambda_2(\beta)^{n-1}}
}
{\zeta(\delta)^{\beta}}
=\,
\frac{
\sum_{n=1}^\infty
\frac{n^{-\delta\, \beta}}{\lambda_2(\beta)^{n}}
}
{ \zeta(\delta)^{\beta}}.
\end{equation}
If we assume that
\begin{equation}\label{prob3}
\limsup_{\beta \to 1}
\frac{\lambda(\beta)}{\lambda_1(\beta)}>1
\qquad
\text{and}
\qquad
\limsup_{\beta \to 1} \frac{\lambda(\beta)}{\lambda_2(\beta)}>1,
\end{equation}
then by (\ref{prob1}) and (\ref{prob2}), respectively,
we have that
\[
	\limsup_{\beta \to 1}
	\frac{ \sum_{n=1}^\infty \frac{n^{-\gamma\, \beta}}
	{\lambda(\beta)^{n}}}{ \zeta(\gamma)^{\beta}} <1.
\qquad
\text{and}
\qquad
\limsup_{\beta \to 1} \,
\frac{ \sum_{n=1}^\infty \frac{n^{-\gamma\,\beta}}
     {\lambda(\beta)^{n}}}{ \zeta(\gamma)^{\beta}}
     <1.
\]
%
%
But this would contradict the equation (\ref{maindouble}).
Therefore, it is not possible
that both inequalities in (\ref{prob3}) holds.

The functions $p_1(\beta)$ and $p_2(\beta)$ are convex,
monotonous decreasing and differentiable in $\beta$,
for $\gamma,\delta>1$ and $t<1$.
It is also known that
for $1<\delta< \gamma<2$, it is true that
$ p_1'(\beta) \to 0$ and  $ p_2'(\beta) \to 0$, when
$\beta\to 1$, $\beta<1$.

From Theorem \ref{ptL} item b) and the L'Hospital Rule
follows that the limit
\begin{equation} \label{lala}
 \lim_{\beta \to 1^{-}}
 \frac{\lambda_1(\beta)}{\lambda_2(\beta)}
 =
 \lim_{\beta \to 1^{-}}
 \frac{c_1 + \frac{1}{\gamma-1}\, \log (1-\beta) }
      {c_2 + \frac{1}{\delta-1}\log (1-\beta) }
 =\frac{ \delta-1}{\gamma-1}<1,
\end{equation}
Therefore, for $\beta$ close to $1$ we have that
\begin{equation} \label{la}
\lambda_2(\beta) > \lambda_1(\beta).
\end{equation}
Since $\lambda(\beta), \lambda_1(\beta),$
and $ \lambda_2(\beta)$
are all convex as functions of $\beta$ it is not
possible that
$\lambda(\beta)\leq \lambda_1(\beta)$,
for $\beta<1$ close to $1$ (otherwise
would contradict   (\ref{maindouble})\,).
Therefore, we get that
$\lambda_2(\beta) \geq \lambda(\beta)\geq \lambda_1(\beta)$.
It follows from  Abel's Theorem that
\begin{equation} \label{twoquo}
\frac{ e^{2\, p(\beta)}}{e^{ p_1(\beta)} \, e^{p_2(\beta)}}
=
\frac{\lambda(\beta)^2}{\lambda_1(\beta)\, \lambda_2(\beta)}
=
\,
\frac{
1 + \sum_{n=2}^\infty \frac{n^{-\gamma\, \beta}}{\lambda(\beta)^{n-1}}
}
{1 + \sum_{n=2}^\infty \frac{n^{-\gamma\, \beta}}{\lambda_1(\beta)^{n-1}} }
\,\,\,
\frac{1 + \sum_{n=2}^\infty \frac{n^{-\delta\, \beta}}{\lambda(\beta)^{n-1}}}
{ 1 + \sum_{n=2}^\infty \frac{n^{-\delta\, \beta}}{\lambda_2(\beta)^{n-1}} }
\to
1,
\end{equation}
when, $\beta\to 1^{-}$,
where $p(\beta) = P(\beta \,g)\equiv P(\beta\, g_{\lambda,\delta}).$

Since $\lambda_1(\beta)>1$ it is not possible that
$\lambda(\beta)\geq \lambda_2(\beta) > \lambda_1(\beta).$
Indeed, if it was true then
the two quotients on the right side of (\ref{twoquo})
would be smaller than one and this is a contradiction.
For similar reason we can not have
$\lambda_2(\beta) >\lambda_1(\beta)\geq \lambda(\beta).$
Therefore, we get that
$\lambda_2(\beta) \geq \lambda(\beta)\geq \lambda_1(\beta).$
From these observations follows that
$p'(\beta)\to 0$, when $\beta\to 1$, $\beta<1$.
Note that when $\beta\to 1$ from below the
right hand side of ($\ref{twoquo}$) goes to
$1$. Therefore,
\[
\lim_{\beta \to 1^{-}} p(\beta)
=
\frac{1}{2} \left(
				\lim_{\beta \to 1^{-}} p_1(\beta)
				+
				\lim_{\beta \to 1^{-}} p_2(\beta)
			\right).
\]
The above equality shows the existence of phase
transition, in the case $\gamma, \delta>2$,
in the sense of lack of differentiability of the pressure.
\end{proof}

\medskip

\begin{theorem} \label{ld2} Consider the Double Hofbauer model.
If $2<\delta< \gamma<3$, then,
\[
\lim_{\beta \to 1^{-}} p(\beta)
=
\frac{1}{2}
\left(
	\lim_{\beta \to 1^{-}} p_1(\beta)
	+
	\lim_{\beta \to 1^{-}} p_2(\beta)
\right).
\]
Moreover,
\begin{equation} \label{plinna}
\lim_{\beta \to 1^{-}} p'(\beta)
\,=\,
\frac{1}{2}\,
\left(
	\frac{\zeta(\gamma)\,
	\log\zeta(\gamma)-\gamma\,\zeta^{\prime}(\gamma)}
	{\gamma \,\zeta^{\prime}(\gamma-1)}
	\,+\,
	\frac{\zeta(\delta)\,
	\log\zeta(\delta)-\delta\,\zeta^{\prime}(\gamma)}
	{\delta\,\zeta^{\prime}(\delta-1)}
\right).
\end{equation}
Since $p(\beta)=0$ for $\beta>1$ there is a lack
of differentiability  of the pressure
at $\beta=1$.
\end{theorem}

\begin{proof}
The proof is analogous to the previous one, but here we have to
use \ref{twoquo} and item c) of Theorem \ref{ptL}
and also to note that the function
\[
\gamma \mapsto
\frac{\zeta(\gamma)\,\log\zeta(\gamma)-\gamma\,\zeta^{\prime}(\gamma)}
{\gamma\,\zeta^{\prime}(\gamma-1)} <0
\]
is monotonous decreasing for $\gamma>2$.
\end{proof}

\section{The Main Eigenfunction}
The eigenfunction of the Ruelle operator (when it exists)
help us to understand important properties of the equilibrium state
of a given potential.
%
%
Continuous eigenfunctions of the Double Hofbauer model do exists
for any $0<\beta<1, \gamma>1$ and $\delta>1$.
Indeed, given real number $a>0$ let $b(\beta)$ be defined such that
\[
b(\beta)
=
\frac{a}{\lambda(\beta)}
\left(
	1 + \sum_{j=2}^\infty
	\frac{j^{-\gamma\,\beta}}{\lambda(\beta)^{j-1}}
\right)
\zeta(\gamma)^{-\beta}.
\]

Note that when $\beta \to 1$ we have that
$b(\beta)\to a$.
From the general result described
in \cite{Wal} applied to our particular case we get that
the eigenfunction $\varphi_{\beta}$, for $\beta<1$ and $n\geq 1$
is given by

\begin{equation}
\varphi_{\beta}(0^n 1...)
=
a
\,\,
\left(
1 + n^{\gamma\,\beta}
\,
\sum_{j=2}^\infty
\frac{(j+ n-1)^{-\gamma\, \beta}\,\,}{\lambda(\beta)^{j-1}}
\right),
\end{equation}
\begin{equation}
\varphi_{\beta}(1^q 0...)
=
b(\beta)\, \,\,
\left(
1 + q^{\delta\, \beta}
\,\sum_{n=2}^\infty
\frac{(n+ q-1)^{-\delta\,\beta}\,\,}{\lambda(\beta)^{n-1}}
\right),
\end{equation}
\[
	\varphi_{\beta}(0^\infty)=a
	\qquad\text{and}\qquad
	\varphi_{\beta}(1^\infty)=b(\beta).
\]
We remark that all
the above series are absolutely convergent because
$\lambda(\beta)>1$
and from the definitions of $\lambda(\beta)$ and $b(\beta)$
follows that
$\varphi_{\beta}(1 0...)= a \, (\lambda(\beta)-1) \zeta(\delta).$

For $\beta=1$ all the above is fine, up to
$\varphi_1(0^\infty) =\infty$
and $ \varphi_1(1^\infty)=\infty.$
In this case $\varphi_1$ is positive but
it has infinite values just in these two points.
Straightforward calculations shown that
\begin{equation}
\varphi_1 (0^n1..)\sim \frac{n}{\gamma-1}
\qquad\text{and}\qquad
\varphi_1 (1^n 0..)\sim\frac{n}{\delta-1}.
\end{equation}
From where it follows that
$\varphi_1 (0^\infty)=\infty$ and $\varphi_1 (1^\infty)=\infty.$
We point out that this will not be a big problem.
\bigskip

Note that if $\gamma>\delta$,
then $\zeta(\gamma)< \zeta(\delta)$, and
$
\varphi_{\beta}(1^\infty)
=
b(\beta)\to
\varphi_1(1^\infty)
=
\,
\varphi_1(0^\infty)$,
when $\beta\to 1^{-}$.
In the symmetric case when $1<\gamma=\delta<2$ it follows from Theorem 1
page 141 in \cite{L1} that
\[
	P(\beta\,g)
	\sim
	c\, (1-\beta)^{\frac{1}{\gamma-1}} +\ \text{high order terms},
	\quad
	\text{when}\ \beta \to 1^{-}.
\]
In the non symmetric case it is not clear how to obtain
the asymptotic expansion of $P(\beta\,g)$ near the critical
point.

We also observe that for $\beta=1$ the pressure vanish and
$\mathcal{L}_g (\varphi_1)=\varphi_1$, therefore
for any $x\in\Omega$ we have that
\[
	\sum_{\sigma(y)=x} e^{g(y)}\, \varphi_1(y)
	=
	\varphi_1(x),
\]
even for $x=0^\infty$ and $x=1^{\infty}$, because
both sides of the above equality are equal to $+\infty$.
By extending the equality below in the obvious way
we obtain for all $x\in \Omega$
\[
\sum_{\sigma(y)=x}
e^{g(y) + \log \varphi_1(y) - \log\varphi_1(\sigma(y))}
=
1.
\]
The function $J\equiv J_g$ defined by $\log J =g + \log \varphi_1 - \log  (\varphi_1
\circ \sigma)$ is called the Jacobian associated to $g$.
Using the above observations we can define an operator $\mathcal{G}$
which sends any continuous function $\psi$ to $\mathcal{G} (\psi)=\phi$,
where $\phi$ is a function defined on $\Omega\setminus \{0^{\infty},1^{\infty}\}$
by
\[
\phi(x)
=
\mathcal{G} (\psi)(x)
=
\sum_{\sigma(y)=x} J_g(y) \, \psi(y).
\]
Since $\mathcal{G}  (1)=1$
the dual operator $\mathcal{G}^*$ acts on the space of probabilities measures
mapping a probability measure $\mu$ on a probability measure
$\mathcal{G}^* (\mu)=\nu$ so that for any continuous function
$\psi$ the following equality holds
\[
\int_{\Omega} \, \psi d\, \mathcal{G}^* (\mu)
=
\int_{\Omega} \psi\, d\nu= \int_{\Omega}\mathcal{G}(\psi) d \mu.
\]
When $\gamma,\delta>2$,
there exists a probability $\mu_1$ (positive in open sets) which is fixed by
$\mathcal{G}^* $.
We can show that in this case this probability measure $\mu_1$ is an equilibrium state
for the potential $g$ (see \cite{Hof}, \cite{L1} and \cite{FL}).
More detailed description of the probability measure $\mu_1$
is given in the next section.

\section{The Eigenprobability}\label{sec2}

By using the Caratheodory Extension Theorem
we can define a finite measure $\nu_1$ on the Borelians of $\Omega$
such that for any natural number $q\geq 1$, we have
\begin{equation}
\nu_1 \left(\overline{0^q1}\right)= q^{-\gamma}
\qquad\text{and}\qquad
\nu_1\left(\overline{1^q0}\right)=q^{-\delta}.
\end{equation}
We chose below the values of $\nu_1(\overline{0})$
and $\nu_1(\overline{1})$ so that the measure $\nu_1$
satisfies $\mathcal{L}_g^*(\nu_1)=\nu_1$.
Note that this fixed point equation is equivalent
to say
that for any function of the form $I_{\overline{0^q1}}$
we have
\[
q^{-\gamma}
=
\int_{\Omega} I_{\overline{0^q1}}\, d \nu_1
=
\int_{\Omega} \mathcal{L}_g (\,I_{\overline{0^q1}}\,) d \nu_1,
\]
and, moreover for any function of the form $I_{\overline{1^q0}}$
we have
\[
q^{-\delta}
=
\int_{\Omega} I_{\overline{1^q0}}\, d \nu_1
=
\int_{\Omega} \mathcal{L}_g (\,I_{\overline{1^q0}}\,) d \nu_1.
\]

Let us compute the last integral above.
First by the definition of the Ruelle Operator we have
\[
\mathcal{L}_g (\,I_{\overline{0^q1}}\,)(x)
=
e^{g(0x)} I_{\overline{0^q1}}(0x)
+e^{g(1x)} I_{\overline{0^q1}}(1x)
=
e^{g(0x)} I_{\overline{0^q1}}(0x).
\]
The last expression is nonzero
if and only if $x\in L_{q-1}$.
Therefore,
\[
\int_{\Omega} \mathcal{L}_g (\,I_{\overline{0^q1}}\,) d \nu_1
=
\int_{\Omega} \,e^{g(0x)} I_{\overline{0^q1}}(0x) d \nu_1(x)
=
\frac{q^{-\gamma} }{(q-1)^{-\gamma} } (q-1)^{\gamma} =q^{-\gamma}.
\]
Analogously we can compute the integral of $I_{\overline{1^q0}}$.
For the probability measure $\nu_1$ to be a eigenprobability for $g$ it must also satisfy
\[
\nu_1 (\overline{0})
=
\int_{\Omega} I_{\overline{0}}\, d \nu_1
=
\int_{\Omega} \mathcal{L}_g (\,I_{\overline{0}}\,) d \nu_1.
\]
By using again the definition of the Ruelle Operator we have
\[
\mathcal{L}_g (\,I_{\overline{0}}\,)(x)
=
e^{g(0x)} I_{\overline{0}}(0x) +e^{g(1x)} I_{\overline{0}}(1x)
=e^{g(0x)} I_{\overline{0}}(0x).
\]
%
%
%
The point $x$ must be in the cylinder $\overline{1}$
or, in the cylinder  $\overline{0}$, which means in some of the sets
$L_n$, $n\geq 1.$
Then,
$$
\nu_1 (\overline{0})
=
\int_{\overline{1}}
e^{g(0x)} I_{\overline{0}}(0x) d \nu_1(x)
+
\sum_{n=1}^\infty
\int_{L_n} e^{g(0x)} I_{\overline{0}}(0x) d \nu_1(x).
$$

Therefore, it follows from the definitions of $g$
and the Lebesgue integral that
\begin{align*}
\nu_1 (\overline{0})
&=
\zeta(\gamma)^{-1} \nu_1(\overline{1}) + 2^{-\gamma}
\nu_1( \overline{01}) + \frac{3^{-\gamma}}{2^{-\gamma}}  \nu_1(\overline{001})+\ldots
\\
&=
\zeta(\gamma)^{-1} \nu_1(\overline{1}) +  2^{-\gamma} + 3^{-\gamma} +\ldots
\\
&=
\zeta(\gamma)^{-1} \nu_1(\overline{1}) + \zeta(\gamma)-1.
\end{align*}
In the same way we obtain
$\nu_1(\overline{1})= \zeta(\delta)^{-1} \nu_1(\overline{0}) +
\zeta(\delta)-1.$
Solving the system we get that
\[
\nu_1 (\overline{0})= \frac{\zeta(\gamma)^{-1} \zeta(\delta)
-\zeta(\gamma)^{-1} +\zeta(\gamma)-1}{1- \zeta(\gamma)^{-1} \zeta(\delta)^{-1}
}
\]
and
\[
\nu_1 (\overline{1})= \frac{\zeta(\delta)^{-1} \zeta(\gamma)
-\zeta(\delta)^{-1} +\zeta(\delta)-1}{1- \zeta(\delta)^{-1} \zeta(\gamma)^{-1}
} .
\]
This measure $\nu_1$ is not necessarily a probability measure.
Since it is a finite measure all we have to do is
multiplying it by a
suitable constant in order to get an eigenprobability.
From now on, we will assume that $\nu_1$ is a probability
measure.
If $\gamma>\delta$ then
some tedious manipulation yields that $\nu_1
(\overline{1})> \nu_1 (\overline{0})$.
This means that the eigenprobability $\nu_1$
gives more mass for regions where the potential is less flat.

Piecing together all the observations on this section we have proved
the following proposition.

\begin{proposition} The above defined probability $\nu_1$ satisfies
$\mathcal{L}_{\,g}^* (\nu_1)=\nu_1.$

\end{proposition}

\medskip

Now we need the following result:

\begin{proposition} \label{w2}
Suppose $g:\Omega\to\mathbb{R}$ is the
continuous function we consider above.
For any $\beta>0$ there exists a eigenprobability $\nu_{\beta}$
and eigenvalue $\Lambda(\beta)$
such that
$\mathcal{L}_{\beta\,g}^* \nu_{\beta} = \Lambda(\beta) \, \nu_{\beta}$.
Moreover,
$\Lambda(\beta)= \lambda(\beta)=\log p(\beta)$,
for all $0\leq \beta<1$.
\end{proposition}

\begin{proof}
Since the potential $\beta\,g$ is continuous we
can define a transformation
$\mathcal{T}$ in the space of probabilities measures
over $\Omega$ such that
$\mathcal{T}(\mu) =\rho$, where
for any continuous function $f$ we have
\[
\int_{\Omega} f \, d\mathcal{T}(\mu)
=
\int_{\Omega} f d \rho
=
\frac{\displaystyle\int_{\Omega}
\mathcal{L}_{\beta\,g}(f) d \mu}
{\displaystyle\int_{\Omega} \mathcal{L}_{\beta\,g} (1) d \mu}.
\]

By the Thichonov-Schauder Theorem there exists
a fixed point $\nu_{\beta}$ for such
$\mathcal{T}$ and we have
$
\Lambda(\beta)
=
\int_{\Omega} \mathcal{L}_{\beta\,g} (1) d \nu_{\beta}
$.
Finally, by using the same reasoning of Section 2
in \cite{PP} we get that
$\lambda(\beta)= \Lambda(\beta)$.

It remains to prove that $\Lambda(\beta)$,
the eigenvalue of the dual operator,
satisfies $\log p(\beta)=  \Lambda(\beta).$
The proof is similar to the one given in \cite{PP}
to the Proposition 3.4. In \cite{PP} the
potential is Lipchitz but the same approach
can be adopted to the  potential we are considering here.
\end{proof}


\bigskip

\section[Phase Transition II. Non-Uniqueness of the Equilibrium State]
{Phase Transition II.\\ Non-Uniqueness of the Equilibrium State}\label{med}

In this section we show the existence of at least two equilibrium probability states
for the potential $\beta g$ at $\beta=1$.
Keeping the notation of the previous section,
for $n\geq 1$, $\beta>1$,
consider the finite measure $\mu_{\beta}$ such that
\[
\mu_{\beta}( L_n) = \nu_t( L_n) \, \varphi_{\beta} (0^n1\ldots),
\qquad
\mu_{\beta}( R_n) =  \nu_{\beta}( R_n)\, \varphi_{\beta} (1^n0\ldots)
\]
$\mu_{\beta} (\overline{0})=\sum_n \mu_{\beta}(L_n)$
and
$\mu_{\beta} (\overline{1})= \sum_n \mu_{\beta}(R_n)$.
This defines the probability measure $\mu_{\beta}$ in a unique way.
For instance, (not normalizing)
we have
$
\mu_1( \overline{01} ) = \varphi_1(01\ldots)
\nu_1  (\overline{01}) = \zeta(\gamma),$
and
$
\mu_1( \overline{001} ) = \varphi_1( 001\ldots)
\nu_1  (\overline{001}) =  \varphi_1( 001\ldots)\, 2^{-\gamma}.$
The bottom line is
\begin{equation} \label{mumu1}
\mu_1 ([0^q1])\sim q^{1-\gamma}\,\, \, \text{and}\,\,\, \mu_1([1^q0])\sim
q^{1-\delta}.
\end{equation}

\begin{proposition}\label{w3}
The above defined probability
$\mu_{\beta}$ is invariant for $\sigma$.

\end{proposition}

\begin{proof}

Given a continuous $f$ we have that
\[
\int_{\Omega} \, f \circ \sigma\, d \mu_{\beta}
=
\int_{\Omega}
\, (f \circ  \sigma)\, \varphi_{\beta}
\, d\nu_{\beta}
=
\int_{\Omega}
\,\frac{1}{\lambda(\beta)}
\mathcal{L}_{\beta g} [\,(f \circ  \sigma)\,\varphi_{\beta}\,]
d \nu_{\beta}.
\]
by using the fact that for any $x\in\Omega$
we have
$
\mathcal{L}_{\beta g} [\,(f \circ  \sigma)\,
\varphi_{\beta}\,] (x)
\,=\,
f(x) \mathcal{L}_{\beta g} [\,\varphi_{\beta}\,] (x)
$,
follows that the r.h.s. above is equal to
\[
\int_{\Omega}
\frac{1}{\lambda(\beta)}\,f\, \mathcal{L}_{\beta g}
(\varphi_{\beta}\,)
\,
d \nu_{\beta}
=
\int_{\Omega} \,\frac{1}{\lambda(\beta)}\,f\, \lambda(\beta)\,
\varphi_{\beta}
\,
d \nu_{\beta}
=
\int_{\Omega} \,f\,
d \mu_{\beta}.
\qedhere
\]
\end{proof}
\medskip

For $\gamma,\delta >2$,
the probability $\mu_1$ is also invariant  because is
the weak limit of invariant probabilities.
Remember that when $ \gamma,\delta<2$
such invariant probility measures do not exist
(the natural candidates would be invariant measure
which maximize pressure and are different from
the Delta Diracs).

\medskip

For $\beta=1$ we have a small problem:
the eigenfunction have the following asymptotic
behavior $\varphi_1 (0^n1..)\sim
n/(\gamma-1)$ and $\varphi_1 (1^n 0..)\sim n/(\delta-1)$ (compare
with expressions (4) and (8) pages 1077 and 1078 in \cite{FL}).
To get a probability measure using the above procedure, we have
to assume that $\gamma,\delta>2$.
In the cases $1<\gamma<2$ or $1<\delta<2$ the above method breaks down and
we do not get a probability measure at $\beta=1$,
just a $\sigma$-finite measure.
By assuming $\gamma,\delta>2$ we have that $\mu_1$ is an
equilibrium probability
(which gives positive mass to open sets at $\beta=1$)
and the entropy of $\mu_1$ is positive.
The lower semicontinuity of the  entropy (and the fact that
$p(\beta)\to 0$) implies that any weak limit
$\nu$ of $\mu_{\beta}$, when $\beta \to 1$, is an
equilibrium probability for $g$. Therefore, for $\gamma, \delta>2$ we have that
$\mu_1$ is an invariant and equilibrium probability for $g$.

An interesting question is: in the case $2>\gamma>\delta>1$ what happens with
the equilibrium probability $\mu_{\beta}$,
when $\beta \to 1^{-}$. We claim that $\mu_{\beta}$
selects $\delta_{1^\infty}$.
Indeed, any probability $\nu$ which is not $\delta_{0^\infty}$ or
$\delta_{1^\infty}$ is such that $\int_{\Omega} g\, d\nu<0$.
Now, we consider  for large and fixed $n$  and $\beta\sim 1$ the quotient
\[
\frac{\mu_{\beta} (\overline{1^n0})}{\mu_{\beta}(\overline{0^n1} )}
\sim
\frac{n^{-\delta}\, \frac{n}{\delta-1 } }
{n^{-\gamma} \, \frac{n}{\gamma-1 }}
\to
\infty.
\]
The above asymptotic expansion means that measure gives
much more mass to the sets closest to the point
$1^\infty$ than those close to the point $0^\infty$.
This shows our claim.


\medskip

The measure $\mu_1$ is a probability measure only if $\delta,\gamma>2$.
In this case by using continuity arguments we can prove
that $\mu_{\beta}$ converges to $\mu_1$, when $\beta\to 1$.
Therefore, $\mu_1$ is selected.
It is interesting to remark that the mass of $\mu_1 (\overline{1})$ is greater
than the mass $\mu_1 (\overline{0})$ using similar arguments we mentioned
above. This describes the influence of the flatness in
the phase transition point.
\medskip

The Theorem \ref{ld1} looks more natural taken in account the above analysis:
for $1<\delta< \gamma<2$  the
left derivative of pressure at $\beta=1$ is zero
(the only selected equilibrium state is $\delta_{1^\infty}$).
On the other hand, for $\gamma=\delta>2$ we have that
the left derivative of pressure at $\beta=1$ is non-zero (see
Theorem 1 page 141 in \cite{L1} for Hofbauer model where the
lack of differentiability is obtained when we
vary the parameter $\gamma$ ).
In this case we also have that
the correlations with respect to $\mu_1$
of some cylinder functions decays polynomially fast
(see for instance
\cite{L1} and Theorem 2.8
in \cite{FL} for the Hofbauer model).
The decay ratio is also explicitly determined as a function of $\gamma$.

\begin{proposition} \label{mainp}

If $\gamma,\delta>2$,
then the above defined  probability measure $\mu_{\beta}$,
$0<\beta<1$ is a fixed point for $\mathcal{G}^*_{\beta}
$, where for any $\psi$
$$
\mathcal{G}_{\beta} (\psi)(x)
=
\sum_{\sigma(y)=x }
J_{\beta , g} (y) \, \psi(y),
$$
and
$
\log J_{\beta,g}
=
\beta\, g \, +\, \log \varphi_{\beta}\, -\,
(\varphi_{\beta} \circ \sigma)-\log \lambda(\beta).
$
This meaning that
$\mathcal{G}^*_{\beta} (\mu_{\beta}) = \mu_{\beta}.$
Moreover, $\log
(\lambda(\beta))= P(\beta g)\equiv p(\beta)$
is monotonous decreasing.
\end{proposition}

\begin{proof}
The proof follows from the fact that if $\nu_{\beta}$ is such that
$\mathcal{L}_{\beta\,g}^*(\nu_{\beta})
=
\lambda(\beta)\nu_{\beta}
$,
then the $\sigma$-invariant
probability
$
\mu_{\beta}=\varphi_\beta\, \nu_\beta$
is fixed for
$
\mathcal{L}^*_{
\beta \,g +
\log \varphi_{\beta} -
\log (\varphi_{\beta} \circ \sigma)
\,-\log \lambda(\beta)
}
$.
Indeed, given $f:\Omega \to \mathbb{R}$ we have
\begin{align*}
\int_{\Omega} \mathcal{G}_{\beta} (f) d\mu_{\beta}
&=
\int_{\Omega}
\mathcal{L}_{ \beta\, g \, +\, \log
\varphi_{\beta}\, -\, (\varphi_{\beta} \circ \sigma )
-\log \lambda(\beta)}\, (f)\,\varphi_{\beta} \,
d \nu_{\beta}
\\
&=
\int_{\Omega}
\frac{1}{\lambda(\beta)}\,\mathcal{L}_{\beta, g} \,
( f\, \varphi_{\beta})
\,\frac{\varphi_{\beta}}{\varphi_{\beta}}  d \nu_{\beta}
=
\int_{\Omega} f \varphi_{\beta}\, d \nu_{\beta}
=
\int_{\Omega} f\, d\mu_{\beta}.
\end{align*}

Note that
$
\log J_{\beta,g}
=
\beta\, g \, +\,
\log \varphi_{\beta}\, -\,
(\varphi_{\beta} \circ \sigma )-
\log \lambda(\beta)
$
is normalized, that is,
$
\mathcal{L}_{
\beta\, g \, +\, \log \varphi_{\beta}\, -\,
(\varphi_{\beta} \circ \sigma )-
\log \lambda(\beta)
}(1)
=1
$,
so we just shown that
\[
\mathcal{L}_{ \beta\, g \, +\, \log \varphi_{\beta}\, -\,
(\varphi_{\beta} \circ \sigma )-\log \lambda(\beta)}^*
(\mu_{\beta}) = \mu_{\beta}.
\]

The Theorem 3.4 in \cite{PP} can be used under our hypothesis.
In particular, we concluded that the Pressure of $\beta g$
is equal to $\log \lambda(\beta)$ (that is, $\log$ of the main eigenvalue).
Since the supremum in the pressure definition is taken over all the
shift invariant probability measures it follows that
\begin{align*}
0
&\geq
P(
\beta\, g \, +\, \log \varphi_{\beta}\,-\,
(\varphi_{\beta} \circ \sigma )-\log \lambda(\beta)
)
\\
&=
P (\beta\, g \, -\log \lambda(\beta) )
\\
&=
\sup_{\mu\in \mathcal{M}_{1}(\sigma)}
\left\{
h(\mu) + \int_{\Omega} [\beta g - \log(\lambda(\beta))]\, d\mu
\right\}
\\
&=
P(\beta\, g) - \log(\lambda(\beta)).
\end{align*}

On the other hand,
when $\mu = \delta_{0^\infty}\in \mathcal{M}_{1}(\sigma)$
we have that $h(\mu) + \beta\,\int_{\Omega} g d\,\mu=0$
and therefore $P(\beta\,g)=\log(\lambda(\beta))$.
This argument shown that there exists
at least two equilibrium states
at the critical point $\beta_c=1$.

The last statement follows from the nonpositivity
of the potential $g$ which implies that
the derivative $p'(\beta)$ is nonpositive for
$\beta<1$.
\end{proof}

As a consequence of this proposition we have the following
corollary.

\begin{corollary} \label{w1}
The pressure of the potential $g$
vanish, that is, $P(g) = 0.$
Moreover, there are at least three equilibrium
probabilities at the phase transition point $\beta=1$
for $\delta,\gamma>2$.
\end{corollary}

\begin{proof}

Note that the potential
$\,g+ \log \varphi_1 - \log (\varphi_1 \circ \sigma)$
is normalized, that is
$\mathcal{L}_{ \,g+ \log \varphi_1 - \log (\varphi_1 \circ \sigma)\,}(1)=1$.
Moreover,
$\mathcal{L}_{\beta\,g}^*(\nu_1)= \nu_1$.
If we consider $\mu_1=\varphi_1\,\nu_1$,
then using the same reasoning of last proposition
we get that
\[
\mathcal{L}^*_{
\,g+ \log \varphi_1 - \log (\varphi_1 \circ \sigma)\,
}(\mu_1)=
\mu_1.
\]
Although $\varphi_1(0^\infty)$ and $\varphi_1(0^\infty)$ are  not defined
the above argument can be applied because $\nu_1$ and $\mu_1$ has no atoms.
By invoking the Theorem 3.4 of \cite{PP} again it follows
that $P(g)\leq 0.$ Indeed, for $\mu = \delta_{0^\infty}$,
we have that $h(\mu) + \int_{\Omega} g d  \,\mu=0$ and therefore $P(g)=0$.
We also have that in any case $\delta_{0^\infty}$ and $\delta_{1^\infty}$
are equilibrium states.
\end{proof}

\begin{remark}
Among the equilibrium probabilities obtained
above one of them assign positive
values to cylinders sets which is the one we
got from the Ruelle Operator.
Therefore, there are at least three ergodic equilibrium states; of
course, convex combinations of them are also equilibrium states. In this case,
as we mention before  (by
continuity arguments) we have that
$\mu_{\beta}$ converges to $\mu_1$, when $\beta\to 1$.
In this case there is selection of the limit probability in
the phase transition point.
\end{remark}

\medskip

\section[{Phase Transitions III. Non-Uniquess of the DLR Measure}]
{Phase Transitions III.\\ Non-Uniquess of the DLR Measure}

In this section we move towards a more probabilistic approach
to obtain the Gibbs measures. The exposition is
based on the Section 2.1 of \cite{Sarig1}.
We refer the reader to \cite{CL} for definitions and results
on DLR probabilities and its relation with Thermodynamic Limit probabilities.

\subsection*{Conditional Expectation: basic facts and notation}

Let $\mathcal{B}$ denote the Borel sigma-algebra on $\Omega= \{0,1\}^\mathbb{N}$
and $\mathcal{X}_n= \sigma^{-n} (\mathcal{B})$, that is,
the $\sigma$-algebra generated by the random variables
$X_n,X_{n+1},\ldots$ on the Bernoulli space, where
$X_n(x)=x_n$ for all $x\in \Omega$.
Fixed a probability measure $m$ defined over $\Omega$ and
given a cylinder set $\overline{a_0a_1\ldots a_{n-1}}$, where $a_j\in \{0,1\}$,
we define
\[
 	\alpha_{ \overline{a_0a_1\ldots a_{n-1}} }(x)
 	=
 	\mathbb{E}_m [ I_{ \overline{a_0a_1\ldots a_{n-1}}  }\, |\, \mathcal{X}_n],
\]
where $\mathbb{E}_m[f|\mathcal{X}_n]$ is the conditional
expectation of $f$ with respect to $m$ given the $\sigma$-algebra
$\mathcal{X}_n$.
From a elementary property of the conditional expectation for
any fixed $b_n,b_{n+1},..,b_r$ we have that
$$
\int_{X_n=b_n,...,X_r=b_r}
I_{ \overline{a_0a_1\ldots a_{n-1}}  }(x) \, dm(x)
=
\int_{X_n=b_n,...,X_r=b_r}
\alpha_{ \overline{a_0a_1\ldots a_{n-1}} }(x) d\, m(x).
$$
In other words
$$
m\left(\overline{a_0a_1\ldots a_{n-1} b_n\ldots b_r} \right)
=
\int_{X_n=b_n,...,X_r=b_r}
\alpha_{ \overline{a_0a_1\ldots a_{n-1}} }(x) d\, m(x).
$$

The measurable functions with respect to $\mathcal{X}_n$
are the functions of the form
$\varphi(\sigma^n (x))$ where $\varphi$ is Borel measurable.
So we can characterize $\alpha_{ \overline{a_0a_1\ldots a_{n-1}} }$
by the following property: for any
$\mathcal{B}$-measurable (or continuous) $\varphi:\Omega\to\mathbb{R}$
$$
\int_{\Omega} \,
\varphi( \sigma^n (x) )
I_{ \overline{a_0a_1\ldots a_{n-1}}  }(x) \, dm(x)
=
\int_{\Omega}
\varphi( \sigma^n (x) )
\alpha_{ \overline{a_0a_1\ldots a_{n-1}} }(x) d\, m(x).
$$

\begin{definition}
Given a potential $\phi$ we say that a probability measure $m$ is a DLR
probability for $\phi$ if for all
$n\in\mathbb{N}$ and any cylinder
set $\overline{x_0x_1\ldots x_{n-1}}$,
we have $m$-almost every
$z=(z_0,z_1,z_2,...)$ that
$$
\mathbb{E}_m ( I_{ \overline{x_0x_1\ldots x_{n-1}}  }\, |\, \mathcal{X}_n)(z)
=
\frac{
	e^{\phi(z)+ \phi(\sigma(z)) + ...+\phi ( \sigma^{n-1}(z) )}
}
{
	\sum_{\sigma^n(z) = \sigma^n (y)}\,
	e^{\phi(y)+ \phi(\sigma(y)) + \ldots +\phi ( \sigma^{n-1}(y) ) }
}.
$$
\end{definition}
The set of all DLR probabilities for
$\phi$ is denoted by $\mathcal{G}^{DLR}(\phi)$.
In general this set is not unique, but for a very
large class of potentials $\mathcal{G}^{DLR}(\beta\phi)$
is unique for $\beta$ large enough, by the Dobrushin Uniqueness
Theorem. So a possible sense of phase transition
is the existence of a inverse temperature $\beta$
so that $\mathcal{G}^{DLR}(\beta\phi)$
posses more than one element.

In the sequel we shown among other things
that the equilibrium probability $\mu_1$ is a
DLR probability for the potential $g$.
We refer the reader to \cite{CL} for more results about DLR probabilities.


\bigskip

\begin{definition}
A continuous positive function $J:\Omega\to\mathbb{R}$
such that for any $x\in\Omega$ we have
$\sum_{\sigma(y)=x} J(y)=1$ is called a Jacobian.
\end{definition}

Here we consider the case where $\phi=\log J$ where $J$ is a Jacobian
(the general case is analyzed in \cite{CL}).
In this case the Ruelle operator $\mathcal{L}_{\log J}$  (for the potential  $\log
J$) is defined as usual for any continuous function $\psi$ by
$$
\mathcal{L}_{\log J} (\psi)(x)
=
\sum_{\sigma(y)=x} J(y) \, \psi(y).
$$

By the definition of a Jacobian we have that $\mathcal{L}_{\log J} (1)=1$.
Remember that
the dual operator $\mathcal{L}_{\log J}^*$ acts on the space
of probability measures. Its action on $\mu$ give us a
probability measure $\mathcal{L}_{\log J}^* (\mu)=\nu$
such that for any continuous function $\psi$
we have
$$
\int_{\Omega} \, \psi d\, \mathcal{L}_{\log J}^* (\mu)
=
\int_{\Omega} \psi\, d\nu
=
\int_{\Omega} \mathcal{L}_{\log J}(\psi) d \mu.$$

 \begin{definition}
A probability $m$ is called a $g-$measure if it is a fixed point for
$\mathcal{L}_{\log J}^*.$
 \end{definition}

\begin{lemma}
For any Jacobian $J$ the operator $\mathcal{L}_{\log J}^*$
has a fixed point which is invariant probability measure for
$\sigma$.
\end{lemma}

\begin{proof}
Since $\mathcal{L}_{\log J} (1)=1$ then $\mathcal{L}_{\log J}^*$
takes probability measures to probability measures.
So existence of a fixed point for $\mathcal{L}_{\log J}^*$
is a straightforward application of the Tychonov-Schauder Theorem.
For the shift invariance of the fixed probability,
see below the Lemma \ref{lema-invariancia-jacobians}.
\end{proof}

\bigskip

If the Jacobian $J$ is in the H\"older class
the probability measure $m$ provided by the above
lemma is unique, that is, the operator
$\mathcal{L}_{\log J}^*$ has only one fixed point.
If $J$ is not in the H\"older class
this is not always true.

In \cite{BK},  \cite{Quas} and \cite{Gallo}
are presented examples where the Jacobian $J$
is {\bf continuous and strictly positive}
and such that there are at least two fixed points
for $\mathcal{L}_{\log J}^*$.
In these cases $\mathcal{G}^*( \log J)$ does not have
cardinality one and so we have phase transition in this sense.
In \cite{Gallo} is presented a criteria for the
existence of more than one $g$-measure.
We will show (see next theorem) how to use this
criteria to exhibit a class of examples, where we have phase
transition in the sense of existence of more than one
DLR probability.

In the examples presented here,
where we have more than one Thermodynamic Limit Gibbs 
probability (see Sections \ref{med} and \ref{TL})
the Jacobian can be zero in some finite subset
of $\Omega$. In this case $\log J$ will be not defined.

For the potentials $g$, with $\gamma,\delta>2$, which were considered in the
previous sections it is also true that the $\nu_1$ is the unique eigenprobability
for the corresponding dual Ruelle $\mathcal{L}_{g}^*.$ Therefore,  $\mathcal{G}^*(g)$ has cardinality one.


\medskip

\begin{proposition}\label{DLR}
Suppose that $J$ is positive and continuous.
If $\mathcal{L}_{\log J}^*(m)=m$ it follows
that $m$ is DLR probability  for the potential $\phi = \log J.$
In other words, $\mathcal{G}^* (\log J)
\subset \mathcal{G}^{DLR}(\log J).$
\end{proposition}

\begin{remark}

Therefore, if we get more than one eigenprobability we get more than one DLR Gibbs probability.

\end{remark}

In what follows we give the
proof of the Proposition \ref{DLR}.
The proof will be divided into three lemmas
which the statements and proofs are given
below.

\begin{lemma}
If $\mathcal{L}_{\log J}^*(m)=m$,
then for any continuous $f$ and $g$ we have
$$
\int_{\Omega} \mathcal{L}_{\log J}(f)\, g\, d m
=
\int_{\Omega} f\, (g \circ \sigma) \, dm.
$$
\end{lemma}

\begin{proof}
Since $\mathcal{L}_{\log J} (f\,(\, g \circ \sigma)) = g\, \mathcal{L}_{\log J}
(f)$, it follows from the definition of the dual operator that
\begin{align*}
\int_{\Omega} f\, (g \circ \sigma) \, dm
=
\int_{\Omega} f\, (g \circ \sigma)
\,d\mathcal{L}_{\log J}^*m
&=
\int_{\Omega}
\mathcal{L}_{\log J}\,(f\, (g \circ \sigma)\,) \, dm
\\
&=
\int_{\Omega}
\mathcal{L}_{\log J}(f)\, g\, d m.
\end{align*}

\end{proof}

\begin{remark}
In the Hilbert space $L^2(\Omega,\mathcal{B},m)$ the dual of $\mathcal{L}_{\log J} $ is the
Koopman operator $g \to \mathcal{K} (g) = g \circ \sigma.$
\end{remark}

\begin{lemma}\label{lema-invariancia-jacobians}
If $\mathcal{L}_{\log J}^*(m)=m$ then $m$ is invariant for
$\sigma$.

\end{lemma}
\begin{proof}
Indeed, given any continuous function $g$ we have
\begin{align*}
\int_{\Omega} g \circ \sigma\, dm
=
\int_{\Omega} (g \circ \sigma)\, d\mathcal{L}_{\log J}^*(m)
&=
\int_{\Omega} \mathcal{L}_{\log J}(g \circ \sigma) \, dm
\\
&=
\int_{\Omega} g\, \mathcal{L}_{\log J}(1) \, dm
=
\int_{\Omega} g \, dm.
\qedhere
\end{align*}
\end{proof}

\begin{lemma}
If $\mathcal{L}_{\log J}^*(m)=m$
then for any continuous function $f:\Omega \to \mathbb{R}$ we have that
$$
\mathbb{E}_m ( f\, |\, \mathcal{X}_n)(x) = \mathcal{L}_{\log J}^n  (f) (\sigma^n
(x)).$$
\end{lemma}

\begin{proof}
Let $g:\Omega \to \mathbb{R}$ be an arbitrary
continuous function. Since for any $n\in\mathbb{N}$
we have
$
\mathcal{L}_{\log J}^n (f\,(\, g \circ \sigma^n))
=
g\, \mathcal{L}_{\log J}^n  (f)
$
it follows that
\begin{align*}
\int_{\Omega} (g\circ \sigma^n (x))\, f(x) d m(x)
&=
\int_{\Omega}  \mathcal{L}_{\log J}^n \,
[\,(g\circ \sigma^n (x))\, f(x)\,] d m(x)
\\
&=
\int_{\Omega} g(x)\, \mathcal{L}_{\log J}^n \,(f) (x) \,d m(x)
\\
&=
\int_{\Omega}
g(\sigma^n (x))\,
\mathcal{L}_{\log J}^n \,(f) (\sigma^n (x)) \,d m(x),
\end{align*}
where in the last equality we use the fact that $m$ is invariant for $\sigma$.
From the previous lemma we get that
\begin{align*}
\mathbb{E}_m ( I_{\overline{a_0a_1\ldots a_{n-1}}}\, |\, \mathcal{X}_n)(x)
&=
\mathcal{L}_{\log J}^n ( I_{\overline{a_0a_1\ldots a_{n-1}}}) (\sigma^n (x))
\\[0.3cm]
&\hspace*{-3.2cm}=
\sum_{ \sigma^n (y) =\sigma^n (x)}
\exp\left(
\log J(y) + \log J (\sigma(y)) +...+ \log J (\sigma^{n-1} (y))
\right)\,
I_{\overline{a_0a_1\ldots a_{n-1}}}(y)
\\[0.3cm]
&\hspace*{-1.2cm}=
\exp\left(
\log J(a_0,a_1,..,a_{n-1}x)
+...+ \log J (\sigma^{n-1} (a_0,a_1,..,a_{n-1}x))
\right)\,
\\[0.3cm]
&\hspace*{-1.2cm}=
\displaystyle
\frac{
\exp\left(
\log J(a_0,a_1,..,a_{n-1}x) +...+ \log J (\sigma^{n-1} (a_0,a_1,..,a_{n-1}x))
\right)
}
{
\sum_{\sigma^n (y) =\sigma^n (x)}
\exp\left(
\log J(y) +...+ \log J (\sigma^{n-1} (y))
\right)
},
\end{align*}
where in the denominator
on last equality we use that $\mathcal{L}_{\log J}(1)=1$.
\end{proof}

Piecing together the three previous lemmas we
have proved the Proposition \ref{DLR}.
When $\phi=\log J$  is
such that $\mathcal{L}_{\log J}^*$ has two invariant probabilities then there
exist two DLR probabilities for
$\phi=\log J$,
meaning that we have phase transition in the DLR sense.

We observe that for the Double Hofbauer potential
the hypothesis of  Theorem 31 required to prove that
$\mathcal{G}^* (g) \subset \mathcal{G}^{DLR}(g)$
in \cite{CL} are also satisfied.

\bigskip
\section[Phase Transition IV. Non Uniqueness of TL via the Renewal equation]
{Phase Transition IV.\\ Non Uniqueness of T.L. - Renewal equation} \label{TL}

In what follows we introduce the so called finite volume
Gibbs measures with a boundary condition $y\in \Omega$ (see \cite{Sarig1}).
For a given  $n\in\mathbb{N}$
consider the probability measure in $(\Omega,\mathscr{F})$
so that for any $F\in\mathscr{F}$, we have

	\[
		\mu_{n}^{y}(F) = \frac{1}{Z_{n}^{y}}
		\sum_{ \substack{ x\in\Omega;\\ \sigma^n(x)=\sigma^n(y) }   }
		1_{F}(x)\exp(- (\,f(x)+ f(\sigma(x))+...+f(\sigma^{n-1}
(x)\,\,)\,)
	\]
where $Z_{n}^y$ is a normalizing factor called
partition function given by
	\[
		Z_{n}^{y}
		=
		\sum_{ \substack{ x\in\Omega;\\ \sigma^n(x)=\sigma^n(y) }   }
		\exp(-  (\,f(x)+ f(\sigma(x))+...+f(\sigma^{n-1} (x)\,\,)).
	\]

A straightforward computations shows
that $\mu_{n}^{y}(\cdot)$ is a probability measure.
This probability measure can be written
in the Ruelle Operator formalism in the
following way:
\begin{equation} \label{TL1}
		\mu_{n}^{y}(F)  =\, \frac{\mathcal{L}_{ f}^n \, (1_F )\,
(\sigma^n (y))}{ \mathcal{L}_{\,f}^n \, (1 )\, (\sigma^n (y)) }
\quad
\text{or}
\quad
\mu_{n}^{y} =\frac{1}{\mathcal{L}_{\, f}^n (1) (\sigma^n(y))} [\,
(\mathcal{L}_{\, f})^{*} \,]^n \, (\delta_{\sigma^n (y)}).
\end{equation}

\begin{definition} For a fixed $y\in \Omega$ any weak limit of the subsequences
$\mu_{n_k,}^y$, when $k\to \infty$ is called Thermodynamic Limit with boundary
conditions $y$. Now we consider the collection of all the Thermodynamic Limits
varying $y\in \Omega$ and take the closed convex hull of this collection.
This set is denoted by $\mathcal{G}^{TL}(f).$
\end{definition}

For a while we assume that the main eigenvalue $\lambda=1$ and $f=\log J$ where $J$ is a
Jacobian. In this case, $ Z^{y}_{n}=1$ for all $n\in\mathbb{N}$
and $y\in\Omega$.  If $f$ is H\"older it is known
(see \cite{PP}) that
for any fixed $y$ we have $\lim_{n \to \infty} \mathcal{L}_{f}^n (I_{[a]}) (y)=
m([a])$, where $m$ is the fixed point for the operator $\mathcal{L}_{f}^*$
(which is the equilibrium state for $\log J$) and $[a]$ is any cylinder set. In
this way the Thermodynamic Limit probability is unique (independent of the
boundary condition).
In other words, if $f=\log J$ is H\"older, then   for any $y\in \Omega$ we have
that
$$
\lim_{n \to \infty} \mu_{n,1}^{y} = \mu,
$$
where $\mu$ is the equilibrium state for $\log J$. In this case
$\mathcal{G}^{TL}(f)$ has cardinality one and there is no phase transition
in the sense of the number of probability measures in $\mathcal{G}^{TL}(f)$.
For a discussion about non-normalized potentials $f$ we refer the reader
to \cite{CL}.

We want to analyze what happens in the case of $f=\log J_g$ (of previous
section) which is not in the H\"older class.
The main question is: for $\beta =1$, the limits $\lim_{n \to \infty} \mu_{n}^{z_1}$
and $\lim_{n \to \infty} \mu_{n}^{z_2}$ can be different ?
The purpose of this section is to answer this question.
Before going into computational details we recall that
in this case we already know (from the previous section)
that the phase transition occurs in the DLR sense.

In this section we follow the results and ideas from \cite{FL}.
From now we denote $\log J$ the normalized Jacobian associated
to $g$, where $g$ is the Double Hofbauer potential.

\begin{definition}

We denote $\mathcal{G}^{TL}_{Per}$  the set of Thermodynamic Limits
probabilities obtained from all periodic points
$y$.
\end{definition}

We will show that the set $\mathcal{G}^{TL}_{Per}(\log J)$ is the convex hull
of the three probability measures
$\delta_{0^\infty}$, $\delta_{1^\infty}$ and $\mu_1$
defined on the Section \ref{med}.
Of course, this implies the existence of phase transition in the
Thermodynamic Limit sense.

An interesting remark is
in \cite{CL} it is shown that for the potential $\log J$
as we are considering here the set
$\mathcal{G}^{TL}(\log J)=\mathcal{G}^{DLR}(\log J)$.
Note that $\mathcal{G}^{*}(\log J)=\{\mu_1\}$
is strictly contained in $\mathcal{G}^{TL}_{Per}(\log J).$
The set of equilibrium states for the pressure is
also equals to $\mathcal{G}^{TL}_{Per}(\log J)$.

Our proof is based on some properties associated to a
kind of Renewal Equation.
Given a sequence $a:\mathbb{N}\to \mathbb{R}$ and a
probability measure $p$ defined on $\mathbb{N}$ we can
ask whether exists or not another sequence $A:\mathbb{N}\to \mathbb{R}$
satisfying the following associated Renewal Equation: for all $q\in \mathbb{N}$
\begin{equation}\label{rene} A(q) = [A(0) p_q + A(1) p_{q-1}  + A(2)\,  p_{q-2}
+...+A(q-2) p_2 +A(q-1) p_1 ]+ a(q).
\end{equation}

If $M= \sum_{q=1}^\infty q\, p_q$ then
the Renewal Theorem (Cap VII Theorem 6.1
in \cite{KT2} and Cap V Theorem 5.1 in \cite{KT1})
claims that
$$ \lim_{q\to \infty} A(q)= \frac{\sum_{q=1}^\infty a(q)}{M}.$$

\medskip

One important feature of the Renewal Theorem
is that the limit value of $A(q)$, as $q\to \infty$, is
provided without knowing the explicit values of the $A(q)$.

We want to investigate the Thermodynamic Limit
$$
\lim_{q \to\infty}\,
\mu_{q}^{y}([a])
=\lim_{q \to\infty}\,
\mathcal{L}_{ \log J}^q (I_{[a]})
(\sigma^q(y)),$$
for different points $y$ in the Bernoulli space
$\Omega$ for an arbitrary cylinder set $[a]$.

We are interested in to find a fixed cylinder $[a]$
for which the above limit does depends on $y\in\Omega$.
If such cylinder do exists one can say that the Double Hofbauer
model has  phase transition in the TL sense.
To accomplish this we will consider the cylinder $[0]$
and periodic points $y$ in the Bernoulli space.
We will show that:
\begin{proposition} \label{q01}
For the Double Hofbauer model
$$
\lim_{q \to\infty}\,
\mu_{q}^{0^\infty}([0]) = 1
\quad\text{and}\quad
\lim_{q \to\infty}\,
\mu_{q}^{1^\infty}([0]) = 0.
$$
\end{proposition}

This shows that there exist more than one probability on the set of Thermodynamic Limits. This means a phase transition in this sense.

The proof of the above proposition will be presented later.
Of course, the aim of this proposition
is to prove the existence of phase transition
for the Double Hofbauer model in the Thermodynamic Limit sense.
But much more can be said in this case and
we want to have the complete picture of this problem
concerning on what happens to this limit when
we consider other points $y\in\Omega$.
In this direction we obtained the following result.
\begin{proposition} \label{dif}
For any periodic points $y$ and $z\in \Omega$
(being not the fixed points) we have
$$
\lim_{q \to\infty}\,
\mu_{q}^{y}([0])
=
\lim_{q \to\infty}\,
\mu_{q}^{z}([0]).
$$
\end{proposition}

As they were stated the above propositions characterize all the
possible Thermodynamic Limit values for the cylinder $[0]$
for any periodic boundary conditions.
It is possible to show more: for any cylinder set $[a]$
the analogous result is true, but its proof requires a more elaborate
calculation and we will not present it here.

\medskip

The proof of the Proposition \ref{dif}
will follow easily from:
\begin{proposition} \label{easy} The limit
$$\lim_{q \to\infty}\,\mathcal{L}_{ \log J}^q (I_{[a]})(y) )$$
is the same for any point of the form  $y=  \underbrace{000...0}_n 1...$ or
$y=\underbrace{111...1}_n 0...$.
\end{proposition}

To facilitate the understanding why the
Proposition \ref{dif} follows from Proposition \ref{easy}
we compute the limit in a simple case.
Let $y$ be a periodic of the form
$y=(011)^\infty=011\, 011\,011...$.
Assuming the Proposition \ref{easy} we have that the following
limits are equal
$$
\lim_{n\to\infty}\,\mathcal{L}_{ \log J}^{3\,n} (I_{[a]})(011)^\infty
=
\lim_{n\to\infty}\,\mathcal{L}_{ \log J}^{3\,n+1} (I_{[a]})(110)^\infty
=
\lim_{n\to\infty}\,\mathcal{L}_{ \log J}^{3\,n+2} (I_{[a]})(101)^\infty.
$$
For this case the statement of the
Proposition \ref{dif} holds true because of
$$
\sigma^{3\,n} (011)^\infty= (011)^\infty
,\,\,\,
\sigma^{3\,n+1} (011)^\infty= (110)^\infty
\,\,
\text{and}\,\,
\sigma^{3\,n+2} (011)^\infty= (101)^\infty.
$$

In order to prove the Proposition \ref{easy}
we will take advantage of some properties of Renewal
Theory. The proof is based on the geometric structure of the
tree graph generated by the pre-images of the point $y$
that one has to consider to compute the value
of the Ruelle operator in the point $y$.
This idea was used in \cite{FL}
but we should remark that the situation
here is more complex. The Figure \ref{grafo-arvore} helps
to understand how the procedure works.
\begin{figure}[h!]
    \centering
   \includegraphics[scale=0.67,angle=0]{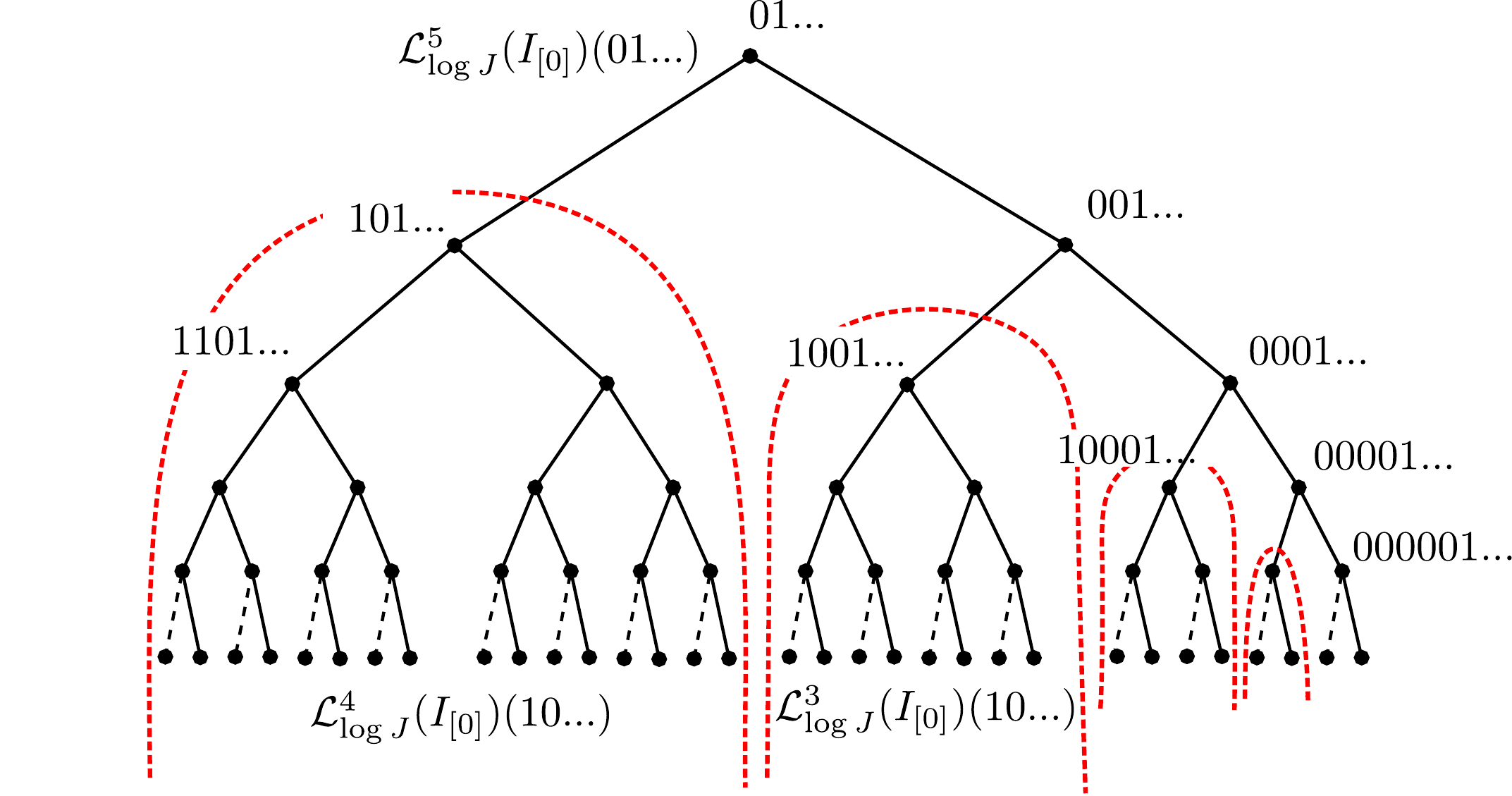}
    \caption{The tree graph representation.}
    \label{grafo-arvore}
\end{figure}

\medskip
\noindent
{\bf Proof of the Proposition \ref{q01}.}
The preimages of $0^\infty$ are $0^\infty$ and $1\,0^\infty$.
By using that $J(10^\infty)=0$ and $J(0^\infty)=1$ we get that
$$
\mu_{1}^{0^\infty}([0])
=\,
\frac{\mathcal{L}_{ \log J} \,(1_{[0]} )\, (\sigma (0^\infty))}
{ \mathcal{L}_{\,\log J} \, (1 )\, (\sigma (0^\infty)) }
=
\mathcal{L}_{ \log J} \, (1_{[0]} )\, (\sigma (0^\infty))
=
\mathcal{L}_{ \log J} \, (1_{[0]} )\, (0^\infty)\,
=
1.
$$
We proceed by induction.
By assuming that $\mathcal{L}_{ \log J}^n  \, (1_{[0]} )\,
(\sigma (0^\infty))=1$ it is easy to see that
\[
\mathcal{L}_{ \log J}^{n+1}  \, (1_{[0]} )\, (\sigma (0^\infty))
=
J(0^\infty)  \mathcal{L}_{ \log J}^n  \, (1_{[0]} )\, (0^\infty)
\, +
J(1 0^\infty)  \mathcal{L}_{ \log J}^n  \, (1_{[0]} )\, (1 0^\infty)
=
1.
\]
Therefore, $\mu_{n}^{0^\infty}([0])\to 1$, when $n \to \infty$.
In the same manner we can see that $\mu_{n}^{1^\infty}([0])\to 0$,
when $n \to \infty.$
\qed
\\
\\
{\bf Proof of the Proposition \ref{easy}}.
The Jacobian for the Double Hofbauer potential
at the inverse temperature $\beta=1$ is such that
$\log J= g + \log \varphi_1 - \log (\varphi_1 \circ \sigma)$.
For simplicity we introduce some notations and
split the computation $\log J(x)$ in six cases:

\begin{itemize}

\item[a)]
for $q\geq 2$ and $x \in L_q$ we have
\begin{align*}
\log J (x)
&=
-\gamma \log \frac{q}{q-1} +
\log \left(
	1+ q^{\gamma}\,\sum_{n=2}^\infty (n+q-1)^{-\gamma}
\right)
\\
&
\qquad\qquad\qquad\qquad\qquad\quad\quad
-\log \left(
	1+(q-1)^{\gamma}\,\sum_{n=2}^\infty (n+q-2)^{-\gamma}
\right)
\\
&:=
-\gamma \log \frac{q}{q-1} + \log r(q)- \log r(q-1).
\end{align*}
\item[b)]
for $q\geq 2$ and $x \in R_q$, we have
\begin{align*}
\log J (x)
&=
-\gamma \log \frac{q}{q-1} +
\log \left(
	1+ q^{\delta}\,\sum_{n=2}^\infty (n+q-1)^{-\delta}
\right)
\\
&
\qquad\qquad\qquad\qquad\qquad\quad\quad
- \log \left(
	1+ (q-1)^{\delta}\,\sum_{n=2}^\infty (n+q-2)^{-\delta}
\right)
\\
&:=
-\delta \log \frac{q}{q-1} + \log s(q)- \log s(q-1).
\end{align*}
\item[c)]
for $x = 0 \overbrace{111...1}^q\,0... \in L_1$ we have
\[
\log J (x)
=
-\log \left(
	1+ q^{\delta}\,\sum_{n=2}^\infty (n+q-1)^{-\delta}
\right)
=
-\log(s(q)).
\]
\item[d)]
for $x = 1 \overbrace{000...0}^q\,1... \in R_1$ we have
\[
\log J (x)
=
-\log \left(
	1+ q^{\gamma}\,\sum_{n=2}^\infty (n+q-1)^{-\gamma}
\right)
=
-\log(r(q)).
\]
\item[e)] for $x=0^\infty$ or $x=1^\infty$ we have
$J(x)=1$.

\item[f)] for $x=10^\infty$ or $x=01^\infty$ we have
$ J(x)=0$.
\end{itemize}
We point out that $\log J(010..)=- \log \zeta(\gamma)=-\log(r(1))$ and $\log
J(101..)=- \log \zeta(\delta)=-\log(s(1)).$
Simple computation shows that
$ J (10^n1..)\sim n^{-1}$
and
$ J (01^n 0..)\sim n^{-1}$.

In this way $J$ is continuous but $\log J (0\, 1^\infty)$ is not defined.

We proceed with the computation
of $\mathcal{L}_{ \log J}^n (I_{[a]}) (y)$
for $y=01\ldots$ and $[a]=[0]$.
Let us stress again that the idea is
to follow \cite{FL} but here we use
the scheme of smaller trees on the right side
(see Figure \ref{grafo-arvore}).
Note that for all $n\in\mathbb{N}$
the value $\mathcal{L}_{ \log J}^n (I_{[0]})(y)$ is
constant and independent of $y\in[01]$
(analogously for $y\in[10]$).
One element in the sum determined by $\mathcal{L}_{ \log J}^q (I_{[0]})
(01...)$ is
$$
\exp\left(
\log J( \underbrace{000...0}_{q+1} \,1) + \log J(
\underbrace{000...0}_{q} \,1)+ \log J( \underbrace{000...0}_{q-1} \,1)
+...+\log J( 0 \,0\,1...)
\right)
$$
which simplifies to
$$
\,\frac{r(q+1)}{r(q)} \left(\frac{q+1}{q}\right)^{-\gamma}
\frac{r(q)}{r(q-1)}\left(\frac{q}{q-1}\right)^{-\gamma} \,...
\left(\frac{2}{1}\right)^{-\gamma}
\frac{r(2)}{\zeta(\gamma)}
=
\frac{(q+1)^{-\gamma}\,\, r(q+1)}{\zeta(\gamma)} .
$$
For $q=1$ this is simply given by
$
\mathcal{L}_{ \log J}^1 (I_{[0]}) (01..)
=2^{-\gamma}r(2)/\zeta(\gamma)
$.
The general term of $\mathcal{L}_{ \log J}^q (I_{[0]})
(01...)$ for $1\leq j \leq q-1$ is given by
the following expression
\begin{multline*}
\biggl[
\exp\biggl(
\log J( 1\underbrace{000...0}_{q-j} \,1)
+\log J( \underbrace{000...0}_{q-j} \,1)
+ \log J( \underbrace{000...0}_{q-j-1}\,1)
\\
+\ldots
+\log J( 0 \,0\,1...)
\biggr)
\biggr]
\,\mathcal{L}_{ \log J}^{j}(I_{[0]})(10...)
\end{multline*}
which is equals to
\[
\frac{(q-j)^{-\gamma}}{\zeta(\gamma)}
\,\mathcal{L}_{ \log J}^{j} (I_{[0]})(10...).
\]
Therefore for any $q\geq 2$ we have that
\begin{multline}
\mathcal{L}_{ \log J}^q (I_{[0]}) (01\ldots)=
\frac{(q-1)^{-\gamma}}{\zeta(\gamma)}
\,\mathcal{L}_{ \log J}^{1} (I_{[0]})(10\ldots)+\ldots
\\
\label{pp1}
+\frac{2^{-\gamma}}{\zeta(\gamma)}\,\mathcal{L}_{
\log J}^{q-2} (I_{[0]}) (10\ldots)+\frac{1}{\zeta(\gamma)}
\,\mathcal{L}_{ \log J}^{q-1} (I_{[0]}) (10\ldots)
\,+ \,\,\,
\frac{(q+1)^{-\gamma} r(q+1)}{\,\,\zeta(\gamma)}  .
\end{multline}

\medskip

The above expression is not exactly a
renewal type of equation because of the powers
of operators involving on it for points of the form
$(01..)$ are described by the powers of the operator
evaluated on points of the form $(10...)$.
We need some more work in order to get an
true renewal equation.
This is the main motivation
for our next step
which is
the computation of
$\mathcal{L}_{ \log J}^n (I_{[a]}) (y)$
for $y=10..$ and $[a]=[0]$ using the scheme of smaller  trees
on the left side.

\medskip
Proceeding as above
and splitting $\mathcal{L}_{ \log J}^q (I_{[0]}) (10..)$
we find the following term
$$
\exp \biggl(
\log J( 0\underbrace{111...1}_{q} \,0) + \log J( \underbrace{111...1}_{q}
\,0)+ \log J( \underbrace{111...1}_{q-1} \,0)+...+\log J( 1 \,1\,0...)
\biggr)
$$
which can be reduced to the expression below by straightforward
computations
\[
\,\frac{1}{s(q)}  \left(\frac{q}{q-1}\right)^{-\delta}
\,\frac{s(q)}{s(q-1)}\left(\frac{q-1}{q-2}\right)^{-\delta}
\,\frac{s(q-1)}{s(q-2)}\ldots
\left(\frac{2}{1}\right)^{-\delta}
\,\frac{s(2)}{s(1)}
=
\frac{1}{\zeta(\delta)}\,q^{-\delta}.
\]
Similarly to the previous step
we can see that for all $n\in\mathbb{N}$
we have that the value of $\mathcal{L}_{ \log J}^n (I_{[0]}) (x)$ is
constant and independent of $x\in[10]$.
Following the geometric picture of the Renewal Equation
it is easy to see that
\begin{multline}\label{del}
\mathcal{L}_{ \log J}^q (I_{[0]}) (10...)
=
\frac{1}{\zeta(\delta)}\,q^{-\delta}+   \frac{1}{\zeta(\delta)}
(q-1)^{-\delta}\,\mathcal{L}_{ \log J}^{1} (I_{[0]}) (01...)+\ldots
\\
+\frac{1}{\zeta(\delta)}
3^{-\delta}\,\mathcal{L}_{ \log J}^{q-3} (I_{[0]})(01..)
+ \frac{2^{-\delta}}{\zeta(\delta)}\,
\mathcal{L}_{ \log J}^{q-2}(I_{[0]}) (01..)
+
\,\frac{1}{\zeta(\delta) }\,\,\mathcal{L}_{ \log
J}^{q-1} (I_{[0]}) (010..).
\end{multline}
Similarly to \eqref{pp1} the above expression
is not exactly a renewal type equation and
to obtain a genuine renewal equation
the idea is to replace (\ref{pp1})  in (\ref{del}).
By doing this we obtain the identity below which
we write down with several terms aiming to help the reader
to identify the pattern emerging from this replacement
\[
\begin{array}{ll}
\mathcal{L}_{ \log J}^q (I_{[0]}) (10..)
=&
\frac{1}{\zeta(\delta)}\,q^{-\delta}+   \frac{1}{\zeta(\delta)}
(q-1)^{-\delta}\, \frac{2^{-\gamma}\,\, r(2)}{\,\,\zeta(\gamma)}
+
\\[0.3cm]
&
\frac{1}{\zeta(\delta)} (q-2)^{-\delta}
\Big[
\,\frac{1}{\zeta(\gamma)}
\,\,\mathcal{L}_{ \log J}^{1} (I_{[0]}) (10..)\,+ \,\,\,\frac{3^{-\gamma}\,\,
r(3)}{\,\,\zeta(\gamma)}
\Big]
+
\\[0.3cm]
&
\frac{1}{\zeta(\delta)}
(q-3)^{-\delta}
\Big[
\,\frac{2^{-\gamma}}{\zeta(\gamma)} \,\,\mathcal{L}_{ \log
J}^{1} (I_{[0]}) (10..)\,+ \frac{1}{\zeta(\gamma)} \,\,\mathcal{L}_{ \log
J}^{2} (I_{[0]}) (10..)\,+
\\[0.3cm]
&
\phantom{
\frac{1}{\zeta(\delta)}
(q-3)^{-\delta}
\Big[
\,\frac{2^{-\gamma}}{\zeta(\gamma)}\
}
\frac{4^{-\gamma}\,\, r(4)}{\,\,\zeta(\gamma)}
\Big]
+\ldots
\\[0.3cm]
&
\frac{1}{\zeta(\delta)} 3^{-\delta}\,
\Big[
\frac{(q-4)^{-\gamma}}{\zeta(\gamma)}\,\mathcal{L}_{ \log J}^{1} (I_{[0]})
(10..)+..+ \frac{1}{\zeta(\gamma)} \,\mathcal{L}_{ \log J}^{q-4} (I_{[0]})
(10..)\,+
\\[0.3cm]
&
\phantom{
\frac{1}{\zeta(\delta)} 3^{-\delta}\,
\Big[
\frac{(q-4)^{-\gamma}}{\zeta(\gamma)}
}
\frac{(q-2)^{-\gamma}\,\, r(q-2)}{\,\,\zeta(\gamma)}
\Big]
+
\\[0.3cm]
&
\frac{1}{\zeta(\delta)}2^{-\delta}\,
\Big[
\frac{(q-3)^{-\gamma}}{\zeta(\gamma)}\,\mathcal{L}_{ \log J}^{1} (I_{[0]})
(10..)+..+ \frac{1}{\zeta(\gamma)} \,\mathcal{L}_{ \log J}^{q-3} (I_{[0]})
(10..)\,+
\\[0.3cm]
&
\phantom{
\frac{1}{\zeta(\delta)}2^{-\delta}\,
\Big[
\frac{(q-3)^{-\gamma}}{\zeta(\gamma)}\,
}
\frac{(q-1)^{-\gamma}\,\, r(q-1)}{\,\,\zeta(\gamma)}
\Big]
\,+
\\[0.3cm]
&
\phantom{2^{-\delta}}
\frac{1}{\zeta(\delta) }
\Big[
\frac{(q-2)^{-\gamma}}{\zeta(\gamma)}\,\mathcal{L}_{ \log J}^{1} (I_{[0]})
(10..)+..+ \,\frac{1}{\zeta(\gamma)}\mathcal{L}_{ \log J}^{q-2} (I_{[0]})
(10..)\,+
\\[0.3cm]
&
\phantom{
2^{-\delta}\frac{1}{\zeta(\delta) }
\Big[
\frac{(q-2)^{-\gamma}}{\zeta(\gamma)}\,
}
\,\frac{q^{-\gamma}\,\, r(q)}{\,\,\zeta(\gamma)}
\Big].
\end{array}
\]
By rearranging the terms of the sum
and make the trivial simplifications
we can see that the above expression
is equal to
\[
\begin{array}{r}
\frac{1}{\zeta(\delta)}\,q^{-\delta}+   \frac{1}{\zeta(\delta)}
(q-1)^{-\delta}\, \frac{2^{-\gamma}\,\, r(2)}{\,\,\zeta(\gamma)}+
...+\frac{1}{\zeta(\delta)} 2^{-\delta}\, \frac{(q-1)^{-\gamma}\,\,
r(q-1)}{\,\,\zeta(\gamma)}  + \frac{1}{\zeta(\delta)}\,\frac{q^{-\gamma}\,\,
r(q)}{\,\,\zeta(\gamma)}\,
+\ \ \ \
\\[0.4cm]
\mathcal{L}_{ \log J}^{1} (I_{[0]})(10..)
\Big[
\frac{(q-2)^{-\delta}}{\zeta(\delta)} \frac{1}{\zeta(\gamma)}  \,+
\frac{(q-3)^{-\delta}}{\zeta(\delta)}\,\frac{2^{-\gamma}}{\zeta(\gamma)}+...+
\,\frac{2^{-\delta}}{\zeta(\delta)} \,\frac{(q-3)^{-\gamma}}{\zeta(\gamma)}+
\,\frac{1}{\zeta(\delta)} \,\frac{(q-2)^{-\gamma}}{\zeta(\gamma)}
\Big]
+
\\[0.4cm]

\mathcal{L}_{ \log J}^{2} (I_{[0]}) (10..)
\Big[
\frac{(q-3)^{-\delta}}{\zeta(\delta)} \,\frac{1}{\zeta(\gamma)}
\,+\frac{(q-4)^{-\delta}}{\zeta(\delta)}\,\frac{2^{-\gamma}}{\zeta(\gamma)}
+...+ \,\frac{2^{-\delta}}{\zeta(\delta)}
\,\frac{(q-4)^{-\gamma}}{\zeta(\gamma)} + \,\frac{1}{\zeta(\delta)}
\,\frac{(q-3)^{-\gamma}}{\zeta(\gamma)}
\Big]+
\\[0.4cm]
\ldots +
\mathcal{L}_{ \log J}^{q-3} (I_{[0]})(10..)
\Big[
\frac{2^{-\delta}}{\zeta(\delta) } \,\frac{1}{\zeta(\gamma)}
\,+\,\frac{1}{\zeta(\delta) }\,\frac{2^{-\gamma}}{\zeta(\gamma) }
\Big]
+
\mathcal{L}_{ \log J}^{q-2} (I_{[0]}) (10..)
\frac{1}{\zeta(\delta) }\frac{1}{\zeta(\gamma)}.
\end{array}
\]
Now we define $A(0)=0$, $p_1=0$ and for $q\geq 2$,
$A(q)=  \mathcal{L}_{ \log J}^{q} (I_{[0]}) (10..)$
and
\begin{equation}
\label{xy1}
\begin{array}{c}
p_q= \, \frac{(q-1)^{-\delta}}{\zeta(\delta)}
\,\frac{1}{\zeta(\gamma)}
\,+\frac{(q-2)^{-\delta}}{\zeta(\delta)}\,\frac{2^{-\gamma}}{\zeta(\gamma)}
+...+ \,\frac{2^{-\delta}}{\zeta(\delta)}
\,\frac{(q-2)^{-\gamma}}{\zeta(\gamma)} + \,\frac{1}{\zeta(\delta)}
\,\frac{(q-1)^{-\gamma}}{\zeta(\gamma)}
\end{array}.
\end{equation}
Let $a(q)$ be the sequence defined for $q\geq 1$ by
\begin{equation} \label{xy2}
\begin{array}{c}
a(q)=
\frac{q^{-\delta}}{\zeta(\delta)}+
\frac{(q-1)^{-\delta}}{\zeta(\delta)}
\frac{2^{-\gamma}\,\,r(2)}{\,\,\zeta(\gamma)}
+...+
\frac{2^{-\delta}}{\zeta(\delta)} \,
\frac{(q-1)^{-\gamma}\,\, r(q-1)}{\,\,\zeta(\gamma)}  +
\frac{1}{\zeta(\delta)}\,\frac{q^{-\gamma}\,\, r(q)}{\,\,\zeta(\gamma)}
\end{array},
\end{equation}
From the definitions we have $a(1)=1/\zeta(\delta)=A(1)$
and
\[
\sum_{j=2}^\infty p_j
=
\sum_{n=1}^\infty
\frac{n^{-\gamma}}{\zeta(\gamma) }
\sum_{n=1}^\infty
\frac{n^{-\delta} }{\zeta(\delta) }
=
1.
\]
By bringing together all the above results
we obtain the following genuine renewal equation
$$
A(q) = [A(0) p_q + A(1) p_{q-1}  + A(2)\,  p_{q-2} +...+A(q-2) p_2 +A(q-1)
p_1 ]+ a(q)
$$
Recalling that $p_1=0$ we have
\begin{equation}\label{rene} A(q) = [A(1) p_{q-1}  + A(2)\,  p_{q-2}
+...+A(q-2) p_2  ]+ a(q).
\end{equation}

\begin{lemma} Let $p_q$ be the sequence above defined.
Then when $q\to \infty$ we have
\begin{equation} \label{maindecayp}
p_q\sim q^{1- \delta- \gamma},
\end{equation}
\end{lemma}
\begin{proof}
To prove this asymptotic behavior it is enough to
prove that the quotient
$p_q/(q\,\,q^{-\gamma}\,\,q^{-\delta})$ has a limit,
when $q\to\infty$. This quotient is explicitly given by
\[
\frac{
\frac{(q-1)^{-\delta}}{\zeta(\delta)} \,\frac{1}{\zeta(\gamma)}
\,+\frac{(q-2)^{-\delta}}{\zeta(\delta)}\,\frac{2^{-\gamma}}{\zeta(\gamma)}
+...+ \,\frac{2^{-\delta}}{\zeta(\delta)}
\,\frac{(q-2)^{-\gamma}}{\zeta(\gamma)} + \,\frac{1}{\zeta(\delta)}
\,\frac{(q-1)^{-\gamma}}{\zeta(\gamma)}\,}
{q\,\,q^{-\gamma}\,\,q^{-\delta}}
\]
which can be rewritten as
\[
\begin{array}{c}
\frac{1}{\zeta(\gamma)\,
\zeta(\delta)}
\sum_{j=1}^{q-1} \,\,
\frac{1}{q}\,
\left(\frac{ j}{q}\right)^{\delta}
\,\,
\left(\frac{q- j}{q}\right)^{\gamma}.
\end{array}
\]
Looking at this expression as Riemann sums we
can guarantee that it has a limit, when $q\to\infty$,
and
$$
\begin{array}{c}
\frac{1}{\zeta(\gamma)\,
\zeta(\delta)}
\sum_{j=1}^{q-1} \,\,
\frac{1}{q}\,
\left(\frac{ j}{q}\right)^{\delta}\,\,
\left(1-\frac{j}{q}\right)^{\gamma}
\ \longrightarrow \
\frac{1}{\zeta(\gamma)\,
\zeta(\delta)}
\int_0^1\, x^{\delta} (1-x)^{\gamma}\, dx
\end{array}
$$
which finish the proof of the lemma.
\end{proof}

By assuming that $\delta< \gamma$ and
using a similar argument as above one
can show that $a(q)\sim q^{2- \delta- \gamma}$.
At this point we have proved that
the hypothesis of the Theorem 6.1 of the
reference \cite{KT2} holds and this theorem allow us to estimate the
limit we are interested in.
It is very important to note that we are also able
to use the result appearing in a remark below the Theorem 6.1
of \cite{KT2} (even though $p_1=0$).
Let us denote $M= \sum_{q=1}^\infty q\, p_q$.
If $\gamma>2$ then $M$ is finite.  The
Renewal Theorem assures that
\begin{equation}
  \lim_{q\to \infty}
  \mathcal{L}_{ \log J}^{q} (I_{[0]}) (10..)
  \,=\,
  \lim_{q\to \infty} A(q)
  =
  \frac{\sum_{q=1}^\infty a(q)}{M}
\end{equation}
One can show that
$
\sum_{q=1}^\infty a(q)
=
1\,+\,
\sum_{j=2}^\infty
\frac{ r(j)\, j^{-\gamma}}{\,\zeta(\gamma)}
$.

Now we proceed to another big step in this proof.
In this step we need to obtain a similar renewal equation for
$\mathcal{L}_{ \log J}^q (I_{[0]}) (01..)$.
In the previous step we have replaced (\ref{pp1}) in (\ref{del}).
Now, we need instead to replace (\ref{del}) in (\ref{pp1}).
Starting as before we write
\[
\begin{array}{ll}
\mathcal{L}_{ \log J}^q (I_{[0]}) (01..)=
&
\frac{(q-1)^{-\gamma}}{\zeta(\gamma)}\,\frac{1}{\zeta(\delta)}
+
\frac{(q-2)^{-\gamma}}{\zeta(\gamma)}
\Big[
\frac{2^{-\delta}}{\zeta(\delta)} +
\frac{1}{\zeta(\delta)}\,\mathcal{L}_{ \log J}^1 (I_{[0]}) (01..)
\Big]
+
\\[0.3cm]
&
\frac{(q-3)^{-\gamma}}{\zeta(\gamma)}
\Big[
\frac{3^{-\delta}}{\zeta(\delta)} +
\frac{2^{-\delta} }{\zeta(\delta)}\,\mathcal{L}_{ \log J}^1 (I_{[0]}) (01..)
\,+ \, \frac{1}{\zeta(\delta)}\,\mathcal{L}_{ \log J}^2 (I_{[0]}) (01..)
\Big]
+
\\[0.3cm]
&
\ldots +
\frac{2^{-\gamma}}{\zeta(\gamma)}
\Big[
\frac{ (q-2)^{-\delta}}{\zeta(\delta)}\,+
\frac{(q-3)^{-\delta}}{\zeta(\delta)}
\,\mathcal{L}_{ \log J}^{1} (I_{[0]})(01..)
+\ldots+
\\[0.3cm]
&
\phantom{
\ldots +
\frac{2^{-\gamma}}{\zeta(\gamma)}
\Big[ \quad
}
\frac{2^{-\delta}}{\zeta(\delta)}\,
\mathcal{L}_{ \log J}^{q-4}(I_{[0]}) (01..)
\,+\,
\frac{1}{\zeta(\delta) }
\,\,\mathcal{L}_{ \log J}^{q-3}(I_{[0]}) (010..)
\Big]
+
\\[0.3cm]
&
+
\frac{1}{\zeta(\gamma)}
\Big[
\frac{(q-1)^{-\delta}}{\zeta(\delta)}
\,+
\frac{(q-2)^{-\delta}}{\zeta(\delta)}
\,\mathcal{L}_{ \log J}^{1} (I_{[0]})(01..)
+\ldots+
\\[0.3cm]
&
\phantom{
\ldots +
\frac{2^{-\gamma}}{\zeta(\gamma)}
\Big[
}
\frac{2^{-\delta}}{\zeta(\delta)}\,\mathcal{L}_{ \log J}^{q-3}
(I_{[0]}) (01..)\,+\,\frac{1}{\zeta(\delta) }\,\,\mathcal{L}_{ \log J}^{q-2}
(I_{[0]}) (010..)
\Big]
\,+
\\[0.3cm]
&
+\frac{(q+1)^{-\gamma}\,\, r(q+1)}{\,\,\zeta(\gamma)}.
\end{array}
\]
By performing obvious simplifications, the above expression
becomes

\[
\begin{array}{r}
\frac{(q-1)^{-\gamma}\,\,}{\,\,\zeta(\gamma)}
\frac{1}{\,\,\zeta(\delta)}  +
\frac{(q-2)^{-\gamma}\,\,}{\,\,\zeta(\gamma)}
\frac{2^{-\delta}}{\,\,\zeta(\delta)}+...+  \frac{1\,\,}{\,\,\zeta(\gamma)}
\frac{(q-1)^{-\delta}}{\,\,\zeta(\delta)}   + \frac{(q+1)^{-\gamma}\,\,
r(q+1)}{\,\,\zeta(\gamma)}
+
\qquad\qquad\qquad\qquad
\\[0.3cm]
\mathcal{L}_{ \log J}^1 (I_{[0]}) (01..)\,
\Big[
\frac{(q-2)^{-\gamma}\,\,}{\,\,\zeta(\gamma)}
\frac{1}{\,\,\zeta(\delta)}
+     \frac{(q-1)^{-\gamma}\,\,}{\,\,\zeta(\gamma)}
\frac{2^{-\delta}}{\,\,\zeta(\delta)}+...+  \frac{1\,\,}{\,\,\zeta(\gamma)}
\frac{(q-2)^{-\delta}}{\,\,\zeta(\delta)}
\Big]
+
\qquad\qquad
\\[0.3cm]
\mathcal{L}_{ \log J}^2 (I_{[0]}) (01..)\,
\Big[
\frac{(q-3)^{-\gamma}\,\,}{\,\,\zeta(\gamma)}
\frac{1}{\,\,\zeta(\delta)}
+     \frac{(q-2)^{-\gamma}\,\,}{\,\,\zeta(\gamma)}
\frac{2^{-\delta}}{\,\,\zeta(\delta)}+...+  \frac{1\,\,}{\,\,\zeta(\gamma)}
\frac{(q-3)^{-\delta}}{\,\,\zeta(\delta)}
\Big]
+
\qquad\quad\
\\[0.3cm]
...+
\mathcal{L}_{ \log J}^{q-3} (I_{[0]}) (01..)\, \{
\frac{2^{-\gamma}\,\,}{\,\,\zeta(\gamma)} \frac{1\,\,}{\,\,\zeta(\delta)}\,+
\frac{1\,\,}{\,\,\zeta(\gamma)} \, \frac{2^{-\delta}\,\,}{\,\,\zeta(\delta)} \}
+
\mathcal{L}_{ \log J}^{q-2} (I_{[0]}) (01..)\,
\frac{1\,\,}{\,\,\zeta(\gamma)} \,\frac{1\,\,}{\,\,\zeta(\delta)} .
\end{array}
\]
Analogously we define $B(0)=0$, $p_1=0$ and
for $q\geq 2$,
$B(q)= \mathcal{L}_{ \log J}^{q} (I_{[0]}) (01..)$.
We also define the following sequence
\begin{equation}
\begin{array}{c}
b(q)=  \frac{(q-1)^{-\gamma}\,\,}{\,\,\zeta(\gamma)}
\frac{1}{\,\,\zeta(\delta)}  +
\frac{(q-2)^{-\gamma}\,\,}{\,\,\zeta(\gamma)}
\frac{2^{-\delta}}{\,\,\zeta(\delta)}+...+  \frac{1\,\,}{\,\,\zeta(\gamma)}
\frac{(q-1)^{-\delta}}{\,\,\zeta(\delta)}   + \frac{(q+1)^{-\gamma}\,\,
r(q+1)}{\,\,\zeta(\gamma)}.
\end{array}
\end{equation}
Note that the first
term of this sequence satisfies
$b(1) = 2^{-\gamma}\,\, r(2)/\zeta(\gamma)=B(1).$
By using the identities established above
we have proved the following renewal equation:
$$
B(q)
=
[B(0) p_q + B(1) p_{q-1}  + B(2)\,  p_{q-2} +...+B(q-2) p_2 +B(q-1)
p_1 ]+ b(q),
$$
where $p_j$ is same sequence we consider in the previous step.
Since $B(0)=0$ we have in fact
\begin{equation}\label{rene2} B(q) = [ B(1) p_{q-1}  + B(2)\,  p_{q-2}
+...+B(q-2) p_2 ]+ b(q).
\end{equation}
Again one can show that $b(q) \sim q^{2-\delta-\gamma}$.
By applying the Renewal Theorem we obtain
\begin{equation}
\lim_{q\to \infty}  \mathcal{L}_{ \log J}^{q} (I_{[0]}) (01..) \,=\,
\lim_{q\to \infty} B(q)= \frac{\sum_{q=1}^\infty b(q)}{M}.
\end{equation}
Since the following equality holds
$$\sum_{q=1}^\infty b(q)=1\,+\,\sum_{j=2}^\infty   \frac{ r(j)\, j^{-\gamma}
}{\,\zeta(\gamma)}$$
we can define a constant $K$ so that
$$
K:=
\frac{\sum_{q=1}^\infty b(q)}{M}
=
\frac{\sum_{q=1}^\infty a(q)}{M}.
$$

The above described procedure allows to obtain
other Thermodynamic Limits.
For instance, if $y=110...$ we have
\begin{align*}
\mathcal{L}_{ \log J}^{q} (I_{[0]}) (10..)
&=
J(010..)\, \mathcal{L}_{ \log J}^{q-1} (I_{[0]})(010..)
+
J(110..) \,\mathcal{L}_{ \log J}^{q-1} (I_{[0]})(110..)
\\
&=
\frac{1}{\zeta(\delta)} \,
\mathcal{L}_{ \log J}^{q-1} (I_{[0]}) (010..)
+
\frac{2^{-\delta}\, s(2)}{\zeta(\delta)}
\mathcal{L}_{ \log J}^{q-1}(I_{[0]}) (110..).
\end{align*}
Taking the limit in $q$ we get
$$
\lim_{q\to \infty}  \mathcal{L}_{ \log J}^{q} (I_{[0]}) (110..)
=
H_1
\qquad
\text{where}
\qquad
H_1
=
\left( K- \frac{1}{\zeta(\delta)}K \right)
\frac{\zeta(\delta)}{2^{-\delta}\, s(2)}
$$
By performing simple algebraic manipulations
we can see that $H_1=K.$
Let us apply the method again but now for $y=001..$.
From the equation
\begin{align*}
\mathcal{L}_{ \log J}^{q} (I_{[0]}) (01..)
&=
 J(110..)\, \mathcal{L}_{ \log J}^{q-1} (I_{[0]}) (101..)
+J(0010..)\,\mathcal{L}_{ \log J}^{q-1} (I_{[0]}) (001..)
\\
&=
\frac{1}{\zeta(\delta)} \,
\mathcal{L}_{ \log J}^{q-1} (I_{[0]}) (010..)+
\frac{2^{-\delta}\, s(2)}{\zeta(\delta)} \,
\mathcal{L}_{ \log J}^{q-1}(I_{[0]}) (110..),
\end{align*}
we get again, by taking the limits in $q$, that
$$
\lim_{q\to \infty}
\mathcal{L}_{ \log J}^{q} (I_{[0]}) (001..)
=
K.
$$
After prove similar result for $y=1110..$
we it is immediate to extend the method
for periodic points by a formal induction.
For this particular case, we have
\begin{align*}
\mathcal{L}_{ \log J}^{q} (I_{[0]}) (110..)
&=
J(0110..)\, \mathcal{L}_{ \log J}^{q-1} (I_{[0]}) (011..)+
J(1110..) \,\mathcal{L}_{ \log J}^{q-1} (I_{[0]})(1110..)
\\
&=
 \frac{1}{s(2)} \, \mathcal{L}_{ \log J}^{q-1} (I_{[0]}) (0110..)+
\frac{3^{-\delta}\, s(3)}{2^{-\delta}\, s(2)}   \mathcal{L}_{ \log J}^{q-1}
(I_{[0]}) (1110..),
\end{align*}
and we get in the same way that
$$
\lim_{q\to \infty}
\mathcal{L}_{ \log J}^{q} (I_{[0]}) (1110..)
=
K.
$$
This finish the proof of Proposition \ref{easy}.

\bigskip

It is worth pointing out that for any $y\in \Omega$ we have that
$$
\mathcal{L}_{ \log J}^{q} (I_{[0]}) (y)
+
\mathcal{L}_{ \log J}^{q} (I_{[1]}) (y)
=
\mathcal{L}_{ \log J}^{q} (1) (y)
=
1.
$$
Therefore for any point $y\in \Omega$
for which the Proposition \ref{easy} applies
we have
$$
\lim_{q\to \infty}  \mathcal{L}_{ \log J}^{q} (I_{[1]}) (y)
=
1-K.
$$

\section{Polynomial Decay of Correlations} \label{decay}

In this section we want to estimate the decay of correlation of the observable
$I_{[0]}$ for the equilibrium probability at
the critical inverse temperature.

That is, we will show that the integral below decays in a polynomial way, with respect to $q$, and also determine its precise asymptotic behavior
\[
\int_{\Omega}
\, (I_{[0]} \circ  \sigma^q)
\,[\,I_{[0]}\, - \mu_1[0]\,] \, d\mu_1
\,\sim
\,q^{2-\delta}.
\]

The technique is similar to the one employed in \cite{FL}
and here we will proofs in full details.
Other known results on polynomial decay of
correlations are presented in
\cite{L1}, \cite{Iso} and \cite{Pol}.

To deduce the polynomial decay
we will need several preliminary estimations.
We begin with the proof of the following asymptotic relation
\[
\mu_1[0]-\mathcal{L}_{ \log J}^{q} (I_{[0]}) (01..)
\sim
q^{2-\delta}.
\]

Define $V_q= \mu_1[0]- B(q)= \mu_1[0]-  \mathcal{L}_{ \log J}^{q} (I_{[0]})
(01..)$. We want to obtain the behavior of $V(q)$, when
$q\to\infty$. From the renewal equation (\ref{rene2})
we get another renewal equation: for
$q\geq 3$
\begin{equation}\label{newren}
V_q = \sum_{j=1}^{q-2} V_j p_{q-j}
+
\left[
\mu_1[0]\, \sum_{j=q}^\infty p_j -b(q)
\right].
\end{equation}

To simplify the notation for $q\geq 3.$
we denote by $K_q$
the last term on the above equality, that is,
$K_q= \mu_1[0]\,\sum_{j=q}^\infty p_j- b(q)$.
Now consider the following formal power series
\[
f(z) = \sum_{j=2}^\infty p_j\, z^j,
\quad
V(z) = \sum_{j=1}^\infty V_j\, z^j,
\quad
\text{and}
\quad
K(z) = \sum_{j=3}^\infty K_j\, z^j.
\]
From the renewal equation (\ref{newren}) we get that
$ V(z) \, f(z) + K(z) = V(z) - V_1 z  - V_2 z^2 .$
Therefore
$$ V(z) = \frac{K(z) + V_1 z + V_2 z^2}{1-f(z)}=  \frac{K(z) + V_1 z + V_2
z^2}{1-z}\, \frac{1-z}{1-f(z)}.$$

For $\delta$ and $\gamma$ large we have that $f(z)$ is differentiable on $z=1$
and the derivative is not zero.
From the previous estimations it is simple to see
that $K_n \sim n^{1-\gamma}.$
Up to a bounded multiplicative constant we get
\[
V(z) \sim \frac{K(z) + V_1 z + V_2 z^2}{1-z}.
\]
from where we obtain (asymptotically)
\begin{align*}
(1-z) \, V(z)
&=
\sum_{j=1}^\infty V_j\, z^j\,
- z \sum_{j=1}^\infty V_j\,z^j
\\
&=
V_1 z + V_2 z^2 + K_3 z^3 + K_4 z^4 + K_5 z^5+...
\\
&=
\sum_{j=3}^\infty K_j\, z^j+ V_1 z + V_2 z^2
\\
&=
K(z)+ V_1(z) + V_2 z^2.
\end{align*}

Note  that
$$\frac{1}{(1-z)} \, (K(z) + V_1 z + v_2 z^2)= (1 + z + z^2 + z^3 +...)  \,\,(V_1 z + V_2 z^2 + K_3 z^3 + K_4 z^4 + ...)= $$
$$ V_1\, z\,+\, (V_1 + V_2)\, z^2 + (V_1 + V_2 +K_3) z^3 +  (V_1 + V_2 +K_3 + K_4) z^4. ...$$

In this way we get the following
recurrence relation: $V_1= 0$ and $V_n= K_3
+ K_4 +..+ K_n + V_2.$
Once $K_n \sim n^{1-\gamma}$
we get that $V_n \sim n^{2-\gamma},$ where we
have assumed for last estimate to holds that $\delta<\gamma.$
This argument complete the proof of
\begin{equation}\label{can1}
\mu_1[0]-  \mathcal{L}_{ \log J}^{q} (I_{[0]}) (01..)\sim n^{2-\gamma}.
\end{equation}
By using the renewal equation (\ref{rene})
and a similar estimation procedure one can prove that
\begin{equation}\label{can2}
\mu_1[0]-  \mathcal{L}_{ \log J}^{q} (I_{[0]}) (10..)\sim n^{3-\delta-\gamma}.
\end{equation}

\bigskip
Now for each $s\geq 2$
we need to evaluate the difference
\[
\mu_1[0]-
\mathcal{L}_{ \log J}^{q} (I_{[0]})(\underbrace{00...0}_{ s}1..).
\]
Let
$
B^s_q=\mathcal{L}_{ \log J}^{q}
(I_{[0]})(\overbrace{00...0}^{s}1..),
$
$
A_t= \mathcal{L}_{ \log J}^{t} (I_{[0]}) (10...)
$
and for $j\geq 1$
$
p^s_j= (s+j-1)^{-\gamma} r(s)^{-1}s^{\gamma}.
$
Using similar arguments presented above (see figure 3 page 1092 \cite{FL})
one can show that for any $s,q\geq 2$
\begin{align}
B^s_q
&=
J (1\underbrace{00...0}_{ s}1..)\, A_{q-1} +  J
(1\underbrace{00...0}_{ s+1}1..) \, A_{q-2} +...+
J (1\underbrace{00...0}_{ s+ q -2}1..)\, A_1
\nonumber\\
&\ \ \ \
+
\frac{ (s+q)^{-\gamma} \, r(s+q)}{s^{-\gamma} \, r(s)}
\nonumber\\
&=
\frac{1}{r(s)} \, A_{q-1} +  \frac{(s+1)^{-\gamma} }{r(s)\, s^{-\gamma} }
\, A_{q-2} +...+
\frac{(s+q-2)^{-\gamma}}{r(s)\,s^{-\gamma}} \, A_1+ \frac{ (s+q)^{-\gamma} \,
r(s+q)  }{s^{-\gamma} \, r(s)}
\nonumber\\
&\label{relo}=
p^s_1\, A_{q-1} +  p^s_2 \, A_{q-2} +...+
p^s_{q-1}\, A_1+ \alpha(q,s).
\end{align}
Note that $ \sum_{j=1}^\infty p^s_j=1.$
Let us now consider for $s\geq 2$ and $n\geq 1$,
the following sequence $V_n^s = \mu_1[0] - B^s_q$.
Recalling the definitions given above if $n=1$ we have
$V_1^s=\mu_1[0] - \frac{ (s+1)^{-\gamma} \, r(s+1)  }{s^{-\gamma} \, r(s)}.$
We also introduce, for $n\geq 1$, the sequences $ K_n=\mu_1[0] - A_n$,
and $U_n^s =\mu_1[0]\,(p^s_n +p^s_{n+1}+...)\, - \alpha(n,s) $.
Note that $\mu_1[0]\,(p^s_n + p^s_{n+1}+...)\sim \frac{s^{1-\gamma}}{\gamma-1}.$
From the equation (\ref{relo})
we deduce that
$$  V_q^s = K_{q-1}\, p_1^s + K_{q-2}\,p_2^s ...+\, K_1\, p^s _{q-1} - U_q^s.$$
The last term behavior is know to be $U_q^s \sim \mu_1 [0] \,
\frac{1}{\gamma-1}\,\frac{(q+s)^{-\gamma+1 }}{s^{-\gamma}\, r(s)}- \frac{
(s+q)^{-\gamma} \, r(s+q)  }{s^{-\gamma} \, r(s)},$
and we know from (\ref{can2}) how to control $K_q$,
that is,  $K_q\sim q^{3-\delta-\gamma}$.
Since $p^s_{q-1}=\frac{(s+q-2)^{-\gamma}}{r(s)\,s^{-\gamma}}  $ we get that for
fixed $s$   the dominant term as a function of $q$ in the right hand side of
the above inequality is $U_q^s\sim \frac{(q+s)^{1-\gamma}}{s^{-\gamma} r(s)}$
because
\begin{equation} \label{can4}
K_{q-1}\, p_1^s + K_{q-2}\,p_2^s ...+\, K_1\, p^s _{q-1} \sim \frac{q^{4-2
\gamma -\delta}}{ s^{1-\gamma}}.
\end{equation}

The above analysis is similar to the one in  Lemma A4 page 1102 in \cite{FL}.

Putting all these estimates together we finally obtain the estimation
\begin{equation} \label{can3}\mu_1[0]-  \mathcal{L}_{ \log J}^{q} (I_{[0]})
(\underbrace{00...0}_{ s}1..) \sim \frac{(q+s)^{1-\gamma}}{s^{-\gamma} r(s)}.
\end{equation}

\bigskip

As mentioned before the Ruelle operator
$\mathcal{L}_{ \log J} $ is the dual (in the
$\mathcal{L}^2(\Omega,\mathcal{B},\mu_1)$ sense) of the Koopman operator $\mathcal{K} (\varphi) =
\varphi \circ \sigma$, using this duality
we get that for fixed $q$ that

\begin{align*}
\int_{\Omega}
(I_{[0]} \circ  \sigma^q)
\,[\,I_{[0]}\, - \mu_1[0]\,] \, d \mu_1
&=
\int  I_{[0]}  \, \mathcal{L}_{ \log J}^{q} [\,I_{[0]}\, - \mu_1[0]\,] \, d
\mu_1
\\
&=
\int_{\Omega}
I_{[0]}(x)
\big[ \mathcal{L}_{ \log J}^{q} \,I_{[0]}(x)\, - \mu_1[0] \big]
\, d \mu_1 (x)
\\
&\hspace*{-1.5cm}=
\sum_{j=1}^\infty\,
\int_{\Omega}
I_{[0]}(\underbrace{00...0}_{ j}1..)
\big[
\mathcal{L}_{ \log J}^{q} \,I_{[0]}(\underbrace{00...0}_{ j}1..)\,
-\mu_1[0]
\big]
\, d \mu_1
\\
&=
\sum_{j=1}^\infty
\mu_1([\underbrace{00...0}_{ j}1]\,
\big[
\mathcal{L}_{ \log
J}^{q} \,I_{[0]}(\underbrace{00...0}_{ j}1..)\, - \mu_1[0]
\big]
\\
&=
\sum_{j=1}^\infty
\, j^{1-\gamma} \,\,\frac{(q+j)^{1-\delta}}{j^{-\gamma}
\, r(q)}\sim q^{2-\delta},
\end{align*}
where in the last equality we used (\ref{can3}) and (\ref{mumu1}).
We point out that the decay of correlations of other observables can also be obtained by variations of the above method.

\section{Potentials on the lattices $\mathbb{N}$ and $\mathbb{Z}$} \label{renor}

In this section we describe the setting where the Ruelle operator provides results for the one-dimensional lattice $\mathbb{Z}$. This is a classical topic and we present it here just for completeness  (see \cite{Rue}, \cite{Kell}, \cite{Bow}, \cite{CL} and \cite{PP}).

The elements of the symbolic space
$\Omega=\{0,1\}^\mathbb{N}$ as in the previous
section are be denoted by $(x_0,x_1,x_2,..)$,
while the elements in two-sided lattice
$\hat{\Omega}=\{0,1\}^\mathbb{Z}$
are denoted by $(...x_{-2},x_{-1}\,|\, x_0,x_1,x_2,...)$.
The action of the left shift $\hat{\sigma}$
on $\hat{\Omega}=\{0,1\}^\mathbb{Z}$
is described by
\(
\sigma (...,x_{-2},x_{-1}\,|\, x_0,x_1,x_2,...)
=
(...,x_{-2},x_{-1}, x_0 \,|\,x_1,x_2,...).
\)

Consider a potential
$\hat{f}:\hat{\Omega}= \{0,1\}^\mathbb{Z}\to \mathbb{R}$
and suppose that we are interested in
finding an equilibrium state $\hat{\mu}$
for the potential $\hat{f}$,
that is, a probability measure on $\hat{\Omega}$ satisfying
$$
	P(\hat{f})
	= \sup_{\nu\in \mathcal{M}(\hat{\sigma})}
	  \left\{ h(\nu) + \int_{\hat{\Omega}} \hat{f}\, d\nu \right\}
	  = h(\hat{\mu}) + \int_{\hat{\Omega}} \hat{f}\, d\hat{\mu}.
$$

Since there is no reason to consider the site $0\in \mathbb{Z}$
a special one in the lattice, then it is natural to
looking for shift-invariant probability measures
among those solving the variational problem.

The crucial observation that relates one sided
and two-sided lattices is the following: the equilibrium
probabilities for $\hat{f}$ and
$g:\{0,1\}^\mathbb{Z}\to\mathbb{R}$, satisfying
$\hat{f}=g+ h - h \circ \hat{\sigma}$ for some continuous
$h:\{0,1\}^\mathbb{Z}\to\mathbb{R}$ are the same.
When such $h$ do exists the above equation
is called a coboundary equation.
The main point is that under some mild regularity assumptions
on the potential $f$, one can get a special $g$
which depends just on future coordinates, that is,
for any pair
$x=(... ,x_{-2},x_{-1}\,|\,x_0,x_1,x_2,...)$
and $y=(... ,y_{-2},y_{-1}\,|\,x_0,x_1,x_2,...) \in \hat{\Omega}$,
we have $g(x)=g(y)$.
By abusing the notation we frequently
write $g(x)= g(x_0,x_1,x_2,...)$,
which allow us to think about $g$ is a
function on $\{0,1\}^\mathbb{N}$.

For the potential $g$ one can apply the formalism of
the Ruelle Operator $\mathcal{L}_g$ to study the
properties of $\mu$ (a probability measure on $\{0,1\}^\mathbb{N}$),
which is the equilibrium state for
$g:\{0,1\}^\mathbb{N}\to\mathbb{R}$.
In those cases it is possible to show that
the equilibrium state for
$\hat{f}$ is the probability measure $\hat{\mu}$
(a probability on $\{0,1\}^\mathbb{Z}$),
which is given by the natural extension of $\mu$
(see \cite{Bow} and \cite{PP}).
Let us elaborate on that:
what we call the natural extension of the shift-invariant probability measure
$\mu$ on $\{0,1\}^\mathbb{N}$
is the probability measure
$\hat{\mu}$ on $\{0,1\}^\mathbb{Z}$ defined in the
following way: for any given cylinder set on the space $\hat{\Omega}$
of the form
$[a_k, a_{k+1},... ,a_{-2},a_{-1}\,|\,a_0,a_1,a_2,..a_{k+n}],$
where $a_j\in\{0,1\}$ and $j\in\{k,k+1,..,k+n\}\subset\mathbb{Z}$, we define
$$
\hat{\mu}\Big([a_k, a_{k+1},... ,a_{-1}\,|\,a_0,a_1,...,a_{k+n}]\Big)
=\mu\Big([a_k, a_{k+1},...,a_{-1},a_0,a_1,..a_{k+n}]\Big),
$$
where $[a_k, a_{k+1},... ,a_{-1},a_0,a_1,...,a_{k+n}]$ is now a cylinder on $\Omega$.
Notice that if $\mu$ is shift-invariant,
then $\hat{\mu}$ is shift-invariant.

Given the potential $\hat{f}:\hat{\Omega} \to \mathbb{R}$,
it is natural to denote  $\mathcal{G}^*(\hat{f})$ as
the set of natural extensions $\hat{\nu}$  of the probabilities
$\nu$ which are eigenprobabilities of the Ruelle
operator $\mathcal{L}_g$ (where $g$ was the associated coboundary).

All the above statements holds true if $\hat f$ is in H\"older
(and then we get that $g$ is also H\"older)
or Walters class (see \cite{Yair-Doan}).
In some cases where $\hat{f}$ (or $g$) is not H\"older,
part of the above formalism also works.
This is indeed the case
of all potentials we considered here.

We shall remind that in order to use the Ruelle
operator formalism we have to work
with potentials $f$ which are defined on the symbolic
space $\Omega$, i.e.,\linebreak
$f:\Omega= \{0,1\}^\mathbb{N}\to \mathbb{R}$.

\end{document}